\documentclass[%
  12pt, 
]{amsart}
  \usepackage{aliascnt}
  \usepackage{mathtools}
  \usepackage{keytheorems}
  \usepackage{float}
  \usepackage{dsfont} 
  \usepackage{amssymb} 
  \usepackage{yhmath} 
  \usepackage[%
      shortlabels,
      inline,
  ]{enumitem} 
  \usepackage{mathrsfs}
  \usepackage[Symbol]{upgreek}
  \usepackage{xfrac}
  \usepackage{accents}
  \usepackage[dvipsnames]{xcolor} 
  \usepackage{interval} 
  \usepackage{mleftright} 
  \usepackage{fixdif}
  \usepackage{leftindex} 
  \usepackage{mathcommand} 
  \usepackage{tikz} 
  \usepackage[all,cmtip]{xy} %
  \usepackage{zref-clever}
  \usepackage[%
      pagebackref, 
  ]{hyperref}
  \usepackage{caption} 
  
  \numberwithin{equation}{section}
  \mathtoolsset{showonlyrefs}
  \let\realItem\item 
\makeatletter
\NewDocumentCommand\myItem{ o }{%
   \IfNoValueTF{#1}%
      {\realItem}
      {\realItem[#1{\MakeLinkTarget[item]{}}]\def\@currentlabel{#1}}
}
\makeatother
\setlist[enumerate,1]{
  before=\let\item\myItem, 
  label = \textup{(\alph*)}, 
  ref = \textup{(\alph*)}
}
  
  \zcsetup{%
      cap,
      abbrev,
      nameinlink,
  }
  \hypersetup{%
      bookmarksnumbered=true,
      colorlinks=true,
      linkcolor=blue,
      citecolor=cyan,
      pdfstartview=FitBH,
  }
  \renewcommand*{\backrefalt}[4]{%
      \ifcase #1 %
        No citations.%
      \or
        $\uparrow$ #2%
      \else
        $\uparrow$ #2%
      \fi
  }
  \providecommand{\noopsort}[1]{} 

  \newkeytheorem{mainthm}[%
      name=Theorem,%
      counter-format=\Alph{mainthm},%
  ]

  \newkeytheorem{%
      theorem,%
      proposition,%
      lemma,%
      corollary,%
      conjecture,%
  }[sharenumber=equation]

  \newkeytheorem{%
      definition,%
      remark,%
      notation,%
      example,%
      warning,%
      construction,%
      convention,%
      question,%
  }[sharenumber=equation,style=definition]

  \newkeytheorem{deflem}[%
      name={Definition/Lemma},%
      sharenumber=equation,%
      style=definition]

      \newkeytheorem{defprop}[%
      name={Definition/Proposition},%
      sharenumber=equation,%
      style=definition]

  \NewDocumentCommand \ListInThm { m O{\textup} O{(\roman{enumi})} O{.} }
    {
      \AtBeginEnvironment{#1}
        {
          \zcsetup{ countertype={enumi=#1} }
          \setlist[enumerate,1]
            {
              label={#2{#3}},
              ref={\csname the#1\endcsname#4{#3}}
            }%
        }
    }

  \ListInThm{theorem}
  \ListInThm{corollary}

\newcommand{\newheaderthm}[1]{
\newtheorem*{__headertheorem__#1}{\zcref{#1}}
}

\newcommand{\startheaderthm}[1]{
\begin{__headertheorem__#1}
}
\newcommand{\finishheaderthm}[1]{
\end{__headertheorem__#1}
}

\newcommand{\headerthmfooter}{}

  \usetikzlibrary{%
      cd,%
      arrows,%
      backgrounds,%
      calc,%
      decorations,%
      decorations.pathmorphing,%
      shapes,%
      tikzmark%
  }
  \tikzcdset{%
      scale cd/.style = %
          {%
            every label/.append style = {scale = #1},%
            cells = {nodes = {scale = #1}}%
          }%
  }

  \NewDocumentMathCommand\txand%
      { O{\quad} }{#1\text{and}#1}
  \NewDocumentMathCommand\txforall%
      { O{\quad} }{#1\text{for all }}
  \NewDocumentMathCommand\txforsome%
      { O{\quad} }{#1\text{for some }}
  \NewDocumentCommand\SetSymbol{o}{\nonscript\:#1\vert\allowbreak\nonscript\:\mathopen{}}  
  \DeclarePairedDelimiterX\Set[1]\{\}
      {#1} 
  \NewDocumentMathCommand\textSet{m}%
      {\Set*{\text{#1}}}
  \DeclarePairedDelimiterX\GSet[1]\langle\rangle
      {#1} 
  \NewDocumentCommand\placeholder{}{\:\cdot\:} 
  \NewDocumentCommand\NewPairedDelimiterS{mmm}{%
      \DeclarePairedDelimiterX{#1}[1]{#2}{#3}%
          {\ifblank{##1}{\placeholder}{##1}}%
  }
  \NewDocumentCommand\NewPairedDelimiterSS{mmmO{,}}{%
      \DeclarePairedDelimiterX{#1}[2]{#2}{#3}%
          {\ifblank{##1}{\placeholder}{##1}%
            #4%
          \ifblank{##2}{\placeholder}{##2}}%
  }
  \NewPairedDelimiterS\abs\lvert\rvert 
  \NewPairedDelimiterS\norm\lVert\rVert 
  \NewPairedDelimiterS\lrangle\langle\rangle 
  \NewPairedDelimiterS\paren\lparen\rparen 
  \NewPairedDelimiterS\lrbrack\lbrack\rbrack 
  \NewPairedDelimiterS\lrbrace\lbrace\rbrace 
  \NewPairedDelimiterS\ceil\lceil\rceil 
  \NewPairedDelimiterS\floor\lfloor\rfloor 
  \NewPairedDelimiterS\bra\langle\vert 
  \NewPairedDelimiterS\ket\vert\rangle 
  \NewPairedDelimiterSS\braket\langle\rangle[%
      \,\delimsize\vert\,\mathopen{}%
  ] 
  \NewPairedDelimiterSS\pairing\langle\rangle
  \NewPairedDelimiterSS\inner\lparen\rparen
  \NewPairedDelimiterSS\Liebracket\lbrack\rbrack
  \DeclarePairedDelimiterX\bracket[3]\langle\rangle%
      {\ifblank{#1}{\placeholder}{#1}%
      \,\delimsize\vert\,\mathopen{}%
      \ifblank{#2}{\placeholder}{#2}
      \,\delimsize\vert\,\mathopen{}%
      \ifblank{#3}{\placeholder}{#3}}%
  \NewDocumentMathCommand\dparen{m}%
      {\lparen\!\lparen{#1}\rparen\!\rparen}
  \NewDocumentMathCommand\dbrack{m}%
      {\lbrack\!\lbrack{#1}\rbrack\!\rbrack}
  \NewDocumentMathCommand\dangle{m}%
      {\langle\!\langle{#1}\rangle\!\rangle}
  \NewDocumentCommand \fun { m e{^_} d() }
      {%
        \operatorname{#1}%
        \IfValueT{#2}{\sp{#2}}%
        \IfValueT{#3}{\sb{#3}}%
        \IfNoValueTF{#4}{}{\mleft(#4\mright)}%
      }
  \LoopCommands{%
      NZQRCKkWM%
  }[#1]{\declaremathcommand#2{\mathbb#1}}
\declaremathcommand\kk{R}
  
  \LoopCommands{%
      {Aut}%
      {an}%
      {Coker}%
      {Cotor}%
      {coker}%
      {Der}%
      {End}%
      {Ext}%
      {Gal}%
      {Gr}%
      {gl}%
      {GL}%
      {gr}%
      {Ho}%
      {Hom}%
      {Id}%
      {id}%
      {Ind}%
      {Ker}%
      {Mor}%
      {Ob}%
      {Pic}%
      {Proj}%
      {pr}%
      {Res}%
      {Sq}%
      {Spc}%
      {Spf}%
      {Spec}%
      {Spm}%
      {supp}%
      {Tor}%
      {tr}%
      {wt}%
      {Tot}%
  }[#1]{\declaremathcommand#2{\fun{#1}}}
    \renewmathcommand\le\leqslant
    \renewmathcommand\ge\geqslant
    \declaremathcommand\one{{\scriptstyle\mathfrak{I}}}
    \declaremathcommand\unit{\mathbb{1}}
    \declaremathcommand\iu{\mathtt{i}}
    \declaremathcommand\e{\mathsf{e}}
    \declaremathcommand\Rf{\operatorname{\mathbf{R}}}
    \declaremathcommand\qbar{\mathscr{q}}
    \declaremathcommand\SG{\mathfrak{S}}
    \declaremathcommand\Gm{\mathbb{G}_{\mathtt{m}}}
    \declaremathcommand\Ga{\mathbb{G}_{\mathtt{a}}}
    \declaremathcommand\vac{\vac}
    \declaremathcommand\cch{\mathscr{c}}
    \declaremathcommand\cfv{\upomega}
    \declaremathcommand\hollowcolon{{}^{\circ}_{\circ}}
    \declaremathcommand\O{\mathscr{O}}
    \declaremathcommand{\H}{\fun{H}}
    \declaremathcommand\et{\text{\'et}}
    \declaremathcommand\pt{\mathsf{p}}
    \declaremathcommand\tpt{\tilde{\mathsf{p}}}
    \declaremathcommand\qt{\mathsf{q}}
    \declaremathcommand\Moduli{\mathcm{M}}
    \declaremathcommand\shHom{\fun{\mathscr{Hom}}}
    \declaremathcommand\V{\mathbb{V}}
    \declaremathcommand\uAut{\operatorname{\underline{Aut}}}
    \declaremathcommand{\Ll}{\mathfrak{L}}
    \declaremathcommand{\G}{\mathbb{G}}
    \declaremathcommand{\vac}{\mathds{1}}
        
    \newmathcommand\Span{\operatorname{span}}
    \newmathcommand\ch{\operatorname{ch}}
    \newmathcommand\Image{\operatorname{Im}}
    
    \newmathcommand\Lie\Liebracket
    \NewDocumentMathCommand\odv{m}%
        {\frac{\d}{\d{#1}}}%
    \NewDocumentMathCommand\pdv{m}%
        {\frac{\partial}{\partial{#1}}}
    \NewDocumentMathCommand\dual{m}
        {\ifblank{#1}{(\:\cdot\:)}{#1}^{\ast}}
    \NewDocumentMathCommand\rldual{m}
        {\ifblank{#1}{(\:\cdot\:)}{#1}^{\dagger}}
    \NewDocumentMathCommand\opp{m}
        {\ifblank{#1}{(\:\cdot\:)}{#1}^{\mathrm{op}}}
    \NewDocumentMathCommand\invo{m}%
        {\prescript{\theta}{}{#1}}%
    \NewDocumentMathCommand\vect{m}%
        {\boldsymbol{#1}}
    \NewDocumentMathCommand\grp{m}%
        {\fun{\mathsf{#1}}}%
    \NewDocumentMathCommand\cat{m}%
        {\operatorname{\mathsf{#1}}}%
      
    \NewDocumentMathCommand\vo{mm}%
        {#1_{(#2)}}%
    \NewDocumentMathCommand\lo{mm}%
        {#1_{[#2]}}%
    \NewDocumentMathCommand\VL{m}%
        {L_{#1}}%

    \NewDocumentMathCommand\Nn{O{\bullet}}%
        {\mathsf{N}^{#1}}
    \NewDocumentMathCommand\NL{O{\bullet}}%
        {\mathsf{N}_{\mathsf{L}}^{#1}}
    \NewDocumentMathCommand\NR{O{\bullet}}%
        {\mathsf{N}_{\mathsf{R}}^{#1}}
    \NewDocumentMathCommand\NLR{O{\bullet}}%
        {\mathsf{N}^{#1}}
    \NewDocumentMathCommand\cNL{O{\bullet}}%
        {\prescript{\mathrm{c}}{}{\mathsf{N}}_{\mathsf{L}}^{#1}}
    \NewDocumentMathCommand\cNR{O{\bullet}}%
        {\prescript{\mathrm{c}}{}{\mathsf{N}}_{\mathsf{R}}^{#1}}
    \NewDocumentMathCommand\cNLR{O{\bullet}}%
        {\prescript{\mathrm{c}}{}{\mathsf{N}}^{#1}}

    \NewDocumentMathCommand\normord{m}%
        {\mathopen{\hollowcolon}\mathinner{#1}\mathclose{\hollowcolon}}%

\newmathcommand\bPhi{\mathbf{\Phi}}

\makeatletter
\newmathcommand{\ostar}{\mathbin{\mathpalette\make@circled\star}}

\newcommand{\make@circled}[2]{%
  \ooalign{$\m@th#1\smallbigcirc{#1}$\cr\hidewidth$\m@th#1#2$\hidewidth\cr}%
}
\newcommand{\smallbigcirc}[1]{%
  \vcenter{\hbox{\scalebox{0.7}{$\m@th#1\bigcirc$}}}%
}

\newmathcommand{\halfstar}{\mathbin{\tikz@halfstar}}
\newcommand{\tikz@halfstar}{%
  \tikzstyle{scorestars}=[star, star points=5, star point ratio=2.25, draw, inner sep=0.2ex, anchor=outer point 3]%
  \begin{tikzpicture}[baseline]%
    \node[scorestars] {};
    \path node[scorestars,fill=black] (s) {} [clip] (s.south west) rectangle (s.north);
  \end{tikzpicture}%
}
\makeatother

\def\cal{\mathcal}
\def\frak{\mathfrak}

\def\eqdef{\stackrel{\textup{def}}{=}}
\def\Im{\operatorname{Im}}

\def\op{{\operatorname{op}}}
\DeclareMathOperator{\im}{im}

\DeclareMathOperator{\ad}{ad}

\DeclareMathOperator{\Mat}{Mat}

\DeclareMathOperator{\Vir}{Vir}

\def\Z{{\mathbb{Z}}}
\def\C{{\mathbb{C}}}
\def\Q{{\mathbb{Q}}}
\def\N{{\mathbb{N}}}

\def\F{{\mathbb{F}}}

\declaremathcommand{\sfL}{{\mathsf{L}}}
\declaremathcommand{\sfR}{{\mathsf{R}}}
\newcommand{\Uu}{\mathit{U}}

\newcommand{\UV}{\mathscr{U}}

\newcommand{\hUu}{\widehat{\Uu}}

\newcommand{\frakL}{\mathfrak{L}}

\newcommand{\LVf}{{\mathfrak{L}}(V)^{\mathsf{f}}}

\newcommand{\PhiL}{\Phi^\mathsf{L}}
\newcommand{\PhiR}{\Phi^\mathsf{R}}

\newcommand{\OmegaL}{\Omega^\mathsf{L}}

\newcommand{\ThetaL}{\Theta^\mathsf{L}}

\newcommand{\Ac}{\mathfrak{A}}

\newcommand{\Aa}{\mathsf{A}}

\usepackage{stackengine}

\makeatletter
\newcommand{\hathat}[1]{%
\begingroup%
  \let\macc@kerna\z@%
  \let\macc@kernb\z@%
  \let\macc@nucleus\@empty%
  \hat{\mathchoice%
    {\raisebox{.3ex}{\vphantom{\ensuremath{\displaystyle #1}}}}%
    {\raisebox{.3ex}{\vphantom{\ensuremath{\textstyle #1}}}}%
    {\raisebox{.3ex}{\vphantom{\ensuremath{\scriptstyle #1}}}}%
    {\raisebox{.2ex}{\vphantom{\ensuremath{\scriptscriptstyle #1}}}}%
    \smash{\hat{#1}}}%
\endgroup%
}

\newcommand{\tildetilde}[1]{%
\begingroup%
  \let\macc@kerna\z@%
  \let\macc@kernb\z@%
  \let\macc@nucleus\@empty%
  \tilde{\mathchoice%
    {\raisebox{.3ex}{\vphantom{\ensuremath{\displaystyle #1}}}}%
    {\raisebox{.3ex}{\vphantom{\ensuremath{\textstyle #1}}}}%
    {\raisebox{.3ex}{\vphantom{\ensuremath{\scriptstyle #1}}}}%
    {\raisebox{.2ex}{\vphantom{\ensuremath{\scriptscriptstyle #1}}}}%
    \smash{\tilde{#1}}}%
\endgroup%
}

\makeatother

\newcommand{\wh}{\widehat}

\newmathcommand\ctensor{\mathbin{\mathop{\widehat{\otimes}}}}

\DeclareMathOperator{\Char}{char}

\begin{document}


\title{Modular Zhu algebra theory and Virasoro vertex algebras}

\begin{abstract}

    We develop the representation theory of vertex algebras over arbitrary commutative rings as part of a relative theory, suitable for various applications such as the study of modular forms, modular tensor categories, and sheaves of coinvariants and conformal blocks. Towards this end, we construct Zhu algebras (and higher analogues) and mode transition algebras over arbitrary rings via the universal enveloping algebra. We prove that these constructions are compatible with base change and give rationality criteria for M\"obius vertex algebras over arbitrary fields satisfying mild assumptions. We apply this framework to study Virasoro vertex operator algebras over arbitrary fields. Using integral forms and base change, we extend the rationality of the discrete series Virasoro VOAs from $\mathbb{C}$ to arbitrary fields of characteristic zero. In positive characteristic, the Virasoro theory exhibits surprising new phenomena: the so-called restricted Virasoro VOA is $C_2$-cofinite and has a semisimple Zhu algebra, but it is not rational. We also show that this VOA is not simple because it admits nontrivial singular vectors in positive characteristic for every central charge. We calculate the singular vector of degree $2p$ in characteristics $p=5,7$ using Mathematica.
\end{abstract}

\date{\today}
\author[Griffin]{Colton Griffin}

\address{David Rittenhouse Lab, University of Pennsylvania, 209 South 33rd Street, Philadelphia, PA 19104-6395}


\setcounter{tocdepth}{1}

\maketitle

\section{Introduction}
Vertex operator algebras (VOAs) are most often studied over $\C$, where their representation theory is closely related to geometric representation theory, modular forms, modular tensor categories, and the geometry of moduli spaces of curves \cite{TUY89,Huang2005}. For example, VOAs give rise to algebraic structures on moduli of curves, including modular forms on $\overline{\mathcal{M}}_{1,n}$ and sheaves of coinvariants and conformal blocks on $\overline{\mathcal{M}}_{g,n}$.
Certain classes of VOAs ($C_2$-cofinite and rational) induce natural systems of cohomology classes, which can be phrased using the language of cohomological field theories \cite{DGT21,DGT22,DGT23}. Since the geometry of $\overline{\mathcal{M}}_{g,n}$ can vary substantially when working over a more general ring than a characteristic 0 field such as $\C$, it is natural from an algebro-geometric perspective to pursue the representation theory of VOAs over any base ring. Moreover, it is desirable to find examples of $C_2$-cofinite and/or rational VOAs in positive characteristic.

The Zhu algebra $\Aa(V)$ of a VOA $V$ plays a central role in the theory of VOAs over $\C$ since it classifies simple admissible modules \cite{Z}. There are criteria for rationality involving the Zhu algebra and its higher analogues \cite{DLM2, McRae_2026}. The recently defined mode transition algebras are an extension of this theory \cite{DGK23}; the existence of certain types of identity elements called \textit{strong units} may be used to detect rationality and establish Morita equivalences between different module categories \cite{DGK24}. Therefore, a primary goal in studying VOAs over arbitrary rings is to develop the theory of the Zhu algebras and the mode transition algebras.

Here, we develop the representation theory for $\Z$-graded vertex algebras over a commutative ring $\kk$ using the associated universal enveloping algebra $\UV$. From this, we give natural generalizations of the higher Zhu algebras $\Aa_n(V)$ from \cite{DLM2} and mode transition algebras $\Ac(V)$ from \cite{DGK23}. We show that $\Aa_n(V)$ also admits the familiar presentation as a quotient of the vertex algebra $V$, generalizing the work of \cite{van_Ekeren_2011} and \cite{HeHigherZhu}. These algebras and other closely related constructions are also compatible with base change along a flat morphism of rings $R\to S$ (\zcref{lem:base change}). As an application, we give criteria for rationality of M\"obius vertex algebras over a field (\zcref{DGK rational equiv}), generalizing the work of \cite{DGK24}, which is stated for VOAs over $\C$.

This relative point of view already leads to new information about Virasoro VOAs, one of the fundamental examples in the subject. We apply our results to integral forms of the Virasoro VOAs and use base change to study their representation theory over arbitrary fields. Using the rationality of discrete series Virasoro VOAs over $\C$ \cite{Wang}, we show that this rationality result similarly holds over any field of characteristic 0 (\zcref{thm: rational over char 0}). Our proof uses our main rationality criterion. This method of pulling the rationality from the proof over $\C$ to a general field of characteristic 0 may be useful in studying $p$-adic vertex algebras \cite{FM22}.

In positive characteristic, the same family of VOAs admits new phenomena in the form of additional singular vectors. While there is a full classification of the singular vectors of the Virasoro Verma modules over $\C$ (or any field of characteristic 0), we are unaware of any such classification in positive characteristic. There exists a characteristic-dependent quotient of the universal Virasoro VOA, denoted $\overline{V}^0_{\Vir}(c,0)_{\mathbb{F}}$, which we refer to as the \textit{restricted Virasoro VOA}. This was first defined in \cite{Li18} under the notation $V^0_{\mathcal{V}\mathrm{ir}}(c,0)_{\mathbb{F}}$.
We show in \zcref{thm:C2-cofinite semisimple Zhu not rational} that over a field $\mathbb{F}$ of characteristic $p\ge 3$, the restricted Virasoro VOA is \textit{not} rational despite being $C_2$-cofinite and having a semisimple Zhu algebra.
This is a surprising result because---to our knowledge---there are \textit{no known examples} of VOAs over $\C$ that admit a semisimple Zhu algebra and are not rational. 
Moreover, there is a criterion for rationality due to \cite{McRae_2026}, which is that a simple, $C_2$-cofinite ($\N$-graded) VOA over $\C$ with semisimple Zhu algebra is rational. It is natural to ask if this criterion holds over any algebraically closed field $\mathbb{F}$.
We show in \zcref{prop:not simple} that $\overline{V}^0_{\Vir}(c,0)_{\mathbb{F}}$ is not simple for all $c\in \mathbb{F}$ and all characteristics $p\ge 3$, suggesting that McRae's criterion may extend to positive characteristic. We discuss the additional singular vectors for $\overline{V}^0_{\Vir}(c,0)_{\mathbb{F}}$ in \zcref{sec:computations}.

There are currently not many known examples of simple, rational VOAs in positive characteristic in the literature \cite{Li15,Li23}. Moreover, there are often ``more'' singular vectors in positive characteristic, resulting in the graded components of the simple quotients being comparably smaller to the simple quotients over $\C$.
This agrees with the general phenomenon that algebraic objects in characteristic $p$ tend to have larger centers (or other subobjects of interest) compared to characteristic 0. Sometimes, as is the case with the affine and Heisenberg VOAs, the simple quotients are \textit{holomorphic} (that is, the adjoint module is the only simple admissible module up to isomorphism). We wonder if the Virasoro VOAs are holomorphic as well.

We believe that the full range of possibilities for VOAs in positive characteristic remains to be seen, and we expect that comparing these objects to VOAs over $\C$ in regards to their geometric applications will be an exciting direction in the field.

\subsection{Plan of the paper}
In \zcref{sec:Basics}, we give an overview of vertex algebras and the relevant modules over an arbitrary ring $\kk$, such as admissible modules. We also give a construction of the universal enveloping algebra  $\UV$ associated to a $\Z$-graded vertex $\kk$-algebra $V$.
In \zcref{sec:contragredient}, we develop the conditions on a vertex algebra $V$ necessary to speak of contragredient duals of graded $V$-modules. We call these \textit{M\"obius vertex algebras}, but the theory was also developed under the name of \textit{$\mathcal{H}$-module vertex algebras} in \cite{Li18}.
In \zcref{sec:Zhu theory}, we develop the theory of \textit{higher Zhu algebras} $\Aa_d(V)$ and \textit{mode transition algebras} $\Ac(V)$ for $\Z$-graded vertex algebras over an arbitrary ring. We study the relationship of these objects to rationality and base change.
In \zcref{sec:Virasoro}, we study VOAs associated to the Virasoro algebra over an arbitrary ring. At the end, we prove our last main results regarding the singular vectors of modular Virasoro VOAs.

There is a Mathematica notebook available \cite{GriffinMathematica} to implement the computations used in \zcref{sec:computations}. The file is an extension of the notebook file developed by Matthew Headrick \cite{HeadrickMathematica} for positive characteristic calculations.\footnote{Some of the calculations involved in \zcref{sec:computations} are too long to reasonably perform by hand, so we include a few different data files with explicit formulas for the Gram matrices. We found that calculating these matrices in Mathematica can take upwards of a few hours.}

\subsection*{Acknowledgements}
We thank Angela Gibney for her comments and guidance in the writing of this work. We are grateful to Harshit Yadav for pointing out a mistake in a prior version of this paper. The author was supported by the NSF GRFP.

\section{Basics}\label{sec:Basics}
In this section, we give the basic definitions and properties of vertex algebras and their associated universal enveloping algebras over an arbitrary ring. These notions are usually stated over $\C$ in the literature, but they can be generalized to an arbitrary base ring.

\subsection{Main definitions}
Here, we define vertex algebras and their modules over an arbitrary ring $\kk$. Given $n\in \Z$ and $k\in \N$, we set
\[(z+w)^n = \sum_{k\ge 0}\binom{n}{k}z^{n-k}w^k,\qquad \binom{n}{k} = \frac{n(n-1)\cdots (n-k+1)}{k!}.\]
\begin{definition}
    Let $\kk$ be a commutative unital ring. A \textit{vertex $\kk$-algebra} is a tuple $(V,\vac,Y)$ consisting of an $\kk$-module $V$, an element $\vac\in V$, and a $\kk$-linear map
    \begin{align*}
        Y(\cdot,z) \colon V&\to \End(V)[\![z,z^{-1}]\!],\\
        a&\mapsto Y(a,z) = \sum_{k\in \Z}a_{(k)}z^{-1-k},
    \end{align*}
    satisfying the following properties:
    \begin{enumerate}[label = (V\arabic*)]
        \item \label{V1} For all $a,b\in V$, there exists $K=K_{a,b}\in \Z$ such that $a_{(k)}b = 0$ if $k\ge K$.
        \item \label{V2} $Y(\vac,z)=\id_V$, or equivalently $\vac_{(k)}a = \delta_{k,-1}a$.
        \item \label{V3} For all $a\in V$, we have $Y(a,z)\vac \in a + zV[\![z]\!]$, or equivalently
        \[a_{(-1)}\vac = a,\quad \text{and}\quad \forall k\ge 0,\ a_{(k)}\vac = 0.\]
        \item \label{V4} For all $a,b,c\in V$, we have the \textit{Jacobi identity}:
        \begin{align*}
            &z^{-1}\delta\left(\frac{x-y}{z}\right)Y(a,x)Y(b,y)c - z^{-1}\delta\left(\frac{-y+x}{z}\right)Y(b,y)Y(a,x)c\\
            &=y^{-1}\delta\left(\frac{x-z}{y}\right)Y(Y(a,z)b,y)c,
        \end{align*}
        where $\delta(z) = \sum_{k\in \Z}z^k$.
    \end{enumerate}
\end{definition}
There are many properties of vertex algebras one can prove, such as the following:
\begin{defprop}\label{prop:VA properties}
    Let $V$ be a vertex $\kk$-algebra. We let
    \[e^{zT} = \sum_{n\ge 0}z^nT^{(n)} \in \End(V)[\![z]\!]\]
    be defined so that $e^{zT}a = Y(a,z)\vac$. That is, $T^{(n)}a = a_{(-1-n)}\vac$. We call $\{T^{(n)}\}_{n\ge 0}$ the \textit{(divided power) translation operators}. Then the following identities hold:
    \begin{enumerate}
        \item (Divided power structure): We have $T^{(0)} = \id_V$ and $e^{zT}e^{wT} = e^{(z+w)T}$, or equivalently for all $n,m\ge 0$
        \[T^{(n)}T^{(m)} = \binom{n+m}{n}T^{(n+m)}.\]
        \item (Translation covariance) For all $a\in V$, we have
        \[e^{wT}Y(a,z)e^{-wT} = Y(e^{wT}a,z) = Y(a,z+w)\eqdef e^{w\partial_z}Y(a,z) = \sum_{k\ge 0}w^k\partial_z^{(k)}Y(a,z).\]
        \item (Skew-symmetry) For all $a,b\in V$, we have
        \[Y(a,z)b = e^{zT}Y(b,-z)a.\]
        \item (Commutator formula) For all $a,b\in V$ and $r,s\in \Z$, we have
        \begin{equation}\label{eq:commutator}
            [a_{(r)},b_{(s)}] = \sum_{k\ge 0}\binom{r}{k}(a_{(k)}b)_{(r+s-k)}.
        \end{equation}
        \item (Iterate/associator formula) For all $a,b\in V$ and $r,s\in \Z$, we have
        \begin{equation}\label{eq:associator}
            (a_{(r)}b)_{(s)} = \sum_{k\ge 0}(-1)^k\binom{r}{k}\left(a_{(r-k)}b_{(s+k)} - (-1)^rb_{(r+s-k)}a_{(k)}\right).
        \end{equation}
    \end{enumerate}
\end{defprop}
One may check that the commutator formula and the associator formula together are equivalent to the Jacobi identity. We will give the component-wise expansion of the Jacobi identity in \zcref{sec equiv Zhu}.
\begin{remark}\label{D-module remark}
    Define $\cal D$ to be the $\kk$-algebra 
    \[\cal D = \kk\left[D^{(n)} \mid n\ge 0\right]\big/\left(D^{(0)} - 1,\, D^{(n)}D^{(m)} - \tbinom{n+m}{n}D^{(n+m)} \mid n,m\ge 0\right).\]
    Then the above proposition states that the operators $\{T^{(n)}\}_{n\ge 0}$ define a representation of $\cal D$ on $V$.
    If $\kk$ is an algebra over the finite field $\F_p$ for a prime $p$, then Lucas' theorem implies that $\cal D$ is generated over $\kk$ by the set $\{D^{(p^n)}\}_{n\ge 0}$, and we have $(D^{(p^n)})^p = 0$ for all $n\ge 0$.
\end{remark}

Throughout the paper, we will require gradings on vertex algebras:
\begin{definition}\label{Z-grading}
    Let $V$ be a vertex $\kk$-algebra. A \textit{$\Z$-grading} on $V$ is a $\Z$-grading on the underlying $\kk$-module
    $V = \bigoplus_{s\in \Z}V_{s}$
    such that:
    \begin{enumerate}
        \item The vacuum element $\vac$ is of degree 0, that is $\vac\in V_{0}$.
        \item For all $a\in V$ homogeneous of degree $\deg(a)$ and $k\in \Z$, the mode $a_{(k)}$ is a homogeneous operator of degree $\deg(a)-1-k$.
    \end{enumerate}
    We say that a $\Z$-grading on $V$ is \textit{lower-truncated} if there exists $K\in \Z$ such that $V_{s} = 0$ unless $s\ge K$. We say that $V$ is \textit{$\N$-graded} if the grading is lower-truncated with $K=0$.

    We say an $\N$-graded vertex algebra $V$ over $\kk$ is \textit{of CFT-type} if the $\N$-grading $V = \bigoplus_{s\ge 0}V_{s}$ is such that $V_{0} = \kk\vac$ and all graded components are finitely generated and projective.
\end{definition}
It is common to require that graded components of vertex algebras be finite-dimensional when working over a field. We have the following replacement when working over an arbitrary ring:
\begin{definition}
    We say that an $\kk$-module $M$ is \textit{rigid} if it is finitely generated and projective over $\kk$. This is equivalent to the evaluation map $M^\lor\otimes M\to \End_{\kk}(M)$ being an isomorphism, where $M^\lor = \Hom_{\kk}(M,\kk)$ is the dual module.
\end{definition}
If $M$ is a rigid $\kk$-module, then $M$ is \textit{reflexive}, meaning that the canonical map $M\to M^{\lor\lor}$ is an isomorphism. The converse generally does not hold, though they are equivalent when $\kk$ is a field.

One way gradings can arise is as part of a larger structure, like a conformal element.
\begin{definition}
    A \textit{vertex operator algebra} (or \textit{VOA}) over a ring $\kk$ is an $\N$-graded vertex $\kk$-algebra $V$ with a distinguished element $\omega\in V_2$, called the \textit{conformal element}, such that the following hold:
    \begin{enumerate}
        \item Each graded part of $V$ is finitely generated and projective over $\kk$.
        \item Denote by $T(z)$ the field corresponding to $\omega$:
        \[T(z) = Y(\omega,z) = \sum_{n\in \Z}\omega_{(n)}z^{-1-n} = \sum_{m\in \Z}L_{m}z^{-2-m}.\]
        We require that the modes $\{L_m\}_{m\in \Z}$ satisfy the commutator relations for the Virasoro algebra such that the central element $\mathbf{c}$ acts as a fixed scalar $c\in \kk$. We refer to $c$ as the \textit{central charge}.
        \item The mode $L_0$ acts as the grading operator on $V$. That is, $L_0v = \deg(v)v$ for all homogeneous $v$.
        \item The mode $L_{-1}$ acts as the first translation operator $T^{(1)}$.
    \end{enumerate}
\end{definition}
We are primarily concerned with the case in which $\kk$ is a PID, in which case a $\kk$-module $M$ is finitely generated and projective if and only if it is free of finite rank. Note that a vertex operator $\kk$-algebra has an obvious action of $\mathfrak{sl}_2(\kk) = \Span_\kk\{L_{-1},L_0,L_1\}$, but this will not be sufficient to define contragredient modules unless $\kk$ is a $\Q$-algebra. We discuss the replacement of this action in \zcref{sec:contragredient}.
\begin{remark}
    Note that the conformal element does not determine the higher-order translation operators $\{T^{(n)}\}_{n\ge 1}$ unless $\kk$ is a $\Q$-algebra. This is one subtle reason that VOAs are often easier to study over $\C$.
\end{remark}

Now we recall modules for vertex algebras.
\begin{definition}
    Let $V$ be a vertex $\kk$-algebra. A \textit{weak $V$-module} is a $\kk$-module $W$ together with a linear map
    \begin{align*}
        Y^W(\cdot,z)\colon V&\to \End(W)[\![z,z^{-1}]\!]\\
        a&\mapsto Y^W(a,z) = \sum_{k\in \Z}a^W_{(k)}z^{-1-k},
    \end{align*}
    satisfying the following properties:
    \begin{enumerate}[label = (M\arabic*)]
        \item For all $a\in V$ and $b\in W$, there exists $K=K_{a,b}\in \Z$ such that $a^W_{(k)}b = 0$ if $k\ge K$.
        \item We have $Y^W(\vac,z) = \id_W$.
        \item For all $a,b\in V$ and $c\in W$, we have
        \begin{align*}
            &z^{-1}\delta\left(\frac{x-y}{z}\right)Y^W(a,x)Y^W(b,y)c - z^{-1}\delta\left(\frac{-y+x}{z}\right)Y^W(b,y)Y^W(a,x)c\\
            &=y^{-1}\delta\left(\frac{x-z}{y}\right)Y^W(Y(a,z)b,y)c.
        \end{align*}
    \end{enumerate}
\end{definition}
One may similarly discuss replacement axioms for $V$-modules, but we will omit this.
\begin{definition}
    Let $V$ be a $\Z$-graded vertex $\kk$-algebra. Let $\Lambda$ be a subset of $\Q$ that is closed under the action of $(\Z,+)$. A \textit{$\Lambda$-grading} on a weak $V$-module $W$ is a $\Lambda$-grading on the underlying $\kk$-module $W = \bigoplus_{\ell\in \Lambda}W_{\ell}$ such that for all $a\in V$ homogeneous of degree $\deg(a)$ and $k\in \Z$, the mode $a^W_{(k)}$ is a homogeneous operator of degree $\deg(a)-1-k$ on $W$. That is, $a_{(k)}^W\cdot W_\ell \subset W_{\ell + \deg(a)-1-k}$ for all $\ell\in \Lambda$.

    We say $W$ is \textit{admissible} if $W$ admits an $\N$-grading, that is, a $\Z$-grading such that $W_{s} = 0$ unless $s\ge 0$. We also fix the following definitions:
    \begin{itemize}
        \item We say that an admissible $V$-module $W$ is \textit{quasi-rigid} if each graded component $W_{\ell}$ is rigid over $\kk$.
        \item We say that an admissible $V$-module $W$ is \textit{quasi-reflexive} if each graded component $W_{\ell}$ is reflexive over $\kk$ (isomorphic to its double dual).
    \end{itemize}
\end{definition}
Every lower-truncated $\Z$-grading on a weak $V$-module $W$ may be shifted so that we may regard it as an admissible $V$-module such that $W_0\neq 0$. The morphisms in the category of admissible $V$-modules are weak $V$-module morphisms. That is, a morphism $W^1\to W^2$ must intertwine the operators $Y^{W^i}(a,z)$, but it need not respect the grading (but in practice it usually does).

\begin{definition}\label{def:rational}
    Let $V$ be a $\Z$-graded vertex $\kk$-algebra. We say that $V$ is \textit{quasi-rational} if every admissible $V$-module is semisimple (a direct sum of simple admissible $V$-modules). We say that $V$ is \textit{rational} if it is quasi-rational and every simple admissible $V$-module is quasi-rigid.
\end{definition}
Being rational is a much more useful condition than being quasi-rational. These conditions are known to be equivalent for a VOA over $\C$ \cite[Theorem 8.1(c)]{DLM1}.

\subsection{The universal enveloping algebra}
Now we discuss the (finite) ancillary Lie algebra associated with a vertex $\kk$-algebra $V$. This Lie algebra and its associated constructions control the module theory for $V$.
\begin{definition}
    Given a vertex $\kk$-algebra $V$, we define the \textit{(finite) ancillary Lie algebra as}
    \begin{align*}
        \LVf &= \frac{V\otimes_{\kk}\kk[t,t^{-1}]}{\left(\Im \nabla^{(n)}\right)_{n> 0}}\cong \frac{V\otimes_{\kk}\kk[t,t^{-1}]}{\left(\Im\left(\id\otimes \partial_{t}^{(n)} - (-1)^nT^{(n)}\otimes \id\right)\right)_{n> 0}},
    \end{align*}
    where $\nabla^{(n)} = \sum_{0\le k\le n}T^{(k)} \otimes \partial_{t}^{(n-k)}$.
\end{definition}
The $\kk$-module isomorphism above follows from iteratively substituting the identity $\nabla^{(n-j)}\equiv 0$ into $\nabla^{(n)}$.
There is a similar definition by replacing $\kk[t,t^{-1}]$ with $\kk(\!(t)\!)$ or $\kk(\!(t^{-1})\!)$ as is done in \cite{DGK23}. However, it is not necessary for our work, so we omit this.

We write $a_{[n]}$ for the image of $a\otimes t^{n}$ in $\LVf$.
The Lie bracket on $\LVf$ is given by
\begin{equation}\label{Lie bracket}
[a_{[n]},b_{[m]}] = \sum_{k\ge 0}\binom{n}{k}(a_{(k)}b)_{[n+m-k]},
\end{equation}
where $a,b\in V$ and $n,m\in \Z$. This bracket satisfies skew-symmetry and the Jacobi identity. Note that Lie algebras are usually required to satisfy $[x,x]=0$ for all $x$. This is equivalent to skew-symmetry if 2 is invertible in $\kk$. All of the rings in our paper will satisfy this, so we will refer to $\LVf$ as a Lie algebra. 
\begin{remark}
	Every commutative unital $\kk$-algebra $A$ with an action of $\mathcal{D}$ from \zcref{D-module remark} is naturally a \textit{commutative} vertex $\kk$-algebra (meaning that $a_{(n)}b = 0$ for all $n\ge 0$ and $a,b\in A$) by setting $a_{(-1-n)}b = (T^{(n)}a)b$ for $a,b\in A$. Therefore, we may identify $V[t,t^{-1}] = V\otimes_\kk\kk[t,t^{-1}]$ as a tensor product of vertex $\kk$-algebras. The translation operators are given by $\nabla^{(n)}$, so $\LVf$ is isomorphic to $D(V[t,t^{-1}])$, where $D(V) = V/(\Im T^{(n)})_{n>0}$ is the Borcherds Lie algebra with the bracket $\{\overline{a},\overline{b}\} = \overline{a_{(0)}b}$.
\end{remark}

\begin{lemma}
    Let $V$ be a $\Z$-graded vertex $\kk$-algebra. Then the Lie algebra $\LVf$ admits a natural $\Z$-grading by setting $\deg(a_{[k]}) = \deg(a)-1-k$ for $a\in V$ homogeneous.
\end{lemma}
\begin{proof}
The claim follows by observing that the Lie bracket in \zcref[noname]{Lie bracket} is degree 0 with respect to the grading defined.
\end{proof}

For what follows, we will repeatedly use the language of \cite[\textsection 2-3]{DGK23}. See the appendix of their paper for more details.
Let $U$ be the universal enveloping algebra of $\LVf$. This naturally inherits a $\Z$-grading from the $\Z$-grading on $\LVf$. The \textit{left and right canonical seminorms} on $U$ are defined to be
\[\NL[m]U = UU_{\le -m} = \sum_{i\le -m}UU_i,\qquad \NR[m]U = U_{\ge m}U = \sum_{i\ge m}U_iU.\]
These both define exhaustive right filtrations of $U$, though their roles are distinct.
\begin{lemma}[\cite[Lemma 2.4.2]{DGK23}]\label{lem:Ann}
One has an identification of left ideals 
\[\NL[m]\Uu = \Uu \Uu_{\le -m} = \Uu \LVf_{\le -m}.\]
Similarly, one has 
\[\NR[m]\Uu = \LVf_{\ge m} \Uu.\]
\end{lemma}
\begin{proof}
The proof uses \cite[Lemma A.2.1]{HeHigherZhu}, which is an identity for vertex algebras over $\C$. However, all of the coefficients are integral, and the proof can be done over any ring. Therefore, the same identity holds over any ring $\kk$.
\end{proof}

We can restrict these seminorms to various filtered and graded parts of these algebras in a canonical way. We write $\NLR[d]_{\bullet}U_{\le p} = (\NLR[d]_{\bullet}U)\cap U_{\le p}$:
\[
    \NL[m]\Uu_{p} = (\Uu \Uu_{\le -m})_{p} = \sum_{j\le -m} \Uu_{p-j}\Uu_{j}, 
    \qquad \NR[m]\Uu_{p} = (\Uu_{\ge m} \Uu)_{p} = \sum_{i\ge m} \Uu_{i}\Uu_{p-i}.
\]
We have $\NR[m+p]\Uu_p = \NL[m] \Uu_p$. Restricting the seminorm to these subspaces, we may define a graded completion of $\Uu$:
Specifically, we define
\[
\hUu_d := \varprojlim_m \dfrac{\Uu_{ d}}{\NL[m]\Uu_{d}} = \varprojlim_m \dfrac{\Uu_{ d}}{\NR[m+d]\Uu_{d}},\qquad \hUu := \bigoplus_d \hUu_d.
\]
We set $J$ to be the two-sided ideal in $\wh{U}$ generated by the Jacobi relations and the vacuum relations. Let $\overline J$ be the respective topological closure of $J$. The resulting quotient algebra
\begin{equation}\label{eq:UV}
    \UV = \hUu/\overline J
\end{equation}
is a graded associative algebra with an almost canonical seminorm.
\begin{definition}
    We call $\UV = \UV(V)$ from \zcref[noname]{eq:UV} the \textit{(finite) universal enveloping algebra} of $V$.
\end{definition}
We have that $\NL[d]\UV_0 = \NR[d]\UV_0$, so we just denote these by $\NLR[d]\UV$. 
Once again, we recommend reading \cite{DGK23} for more details on the above constructions and statements.

\begin{lemma}\label{lem:weak V mods equiv}
    Let $V$ be a $\Z$-graded vertex $\kk$-algebra. The category of weak $V$-modules is equivalent to the category of continuous left $\UV$-modules.
\end{lemma}
This statement is already known over $\C$, but the proof is the same over any ring $\kk$.
\begin{remark}
    The definition of $V$-module that most people work with is better interpreted as a ``left $V$-module.'' In general, there is no equivalence between the categories of left and right $\UV$-modules, so there could be an inequivalent notion of a ``right $V$-module.'' However, in most cases (including ours) the universal enveloping algebra $\UV$ admits an involution, so we may regard left and right $V$-modules as the same. For example, $\UV$ admits an involution when $V$ is a VOA over $\C$. We will describe a generalization of this involution in the next section.
\end{remark}

\section{Contragredient duals and symmetric bilinear forms}\label{sec:contragredient}
The purpose of this section is to give a generalization of some of the constructions of \cite{Li18} from an algebraically closed base field to any base ring. Most of our statements here will be given without proof since one may apply the original paper's arguments verbatim.

\subsection{Main definitions}
Consider the Lie algebra $\mathfrak{sl}_2(\kk)$ over a ring $\kk$ with the basis $\{L_{-1},L_0,L_1\}$, satisfying the relations
\[[L_1,L_{-1}] = 2L_0,\quad [L_0,L_{\pm 1}] = \mp L_{\pm 1}.\]
Given $n\in \N$, we set
\[L_{\pm 1}^{(n)} = \frac{L_{\pm 1}^n}{n!},\qquad L_0^{(n)} = \binom{-2L_0}{n}.\]
We define $U(\mathfrak{sl}_2)_\Z$ to be the noncommutative ring generated by the symbols $\{L_{\pm 1}^{(n)},L_{0}^{(n)}\}_{n\in \N}$ with the relations
\begin{align*}
    L_{\pm 1}^{(0)} &= L_0^{(0)} = 1,\\
    L_0^{(m)}L_0^{(n)} &= \sum_{j=0}^m\binom{m}{j}\binom{n+j}{m}L_0^{(n+j)},\\
    L_{\pm 1}^{(m)}L_{\pm 1}^{(n)} &= \binom{m+n}{n}L_{\pm 1}^{(m+n)},\\
    L_0^{(m)}L_{\pm 1}^{(n)} &= L_{\pm 1}^{(n)}\binom{-2L_0\pm 2n}{m} = \sum_{i=0}^m\binom{\pm 2n}{i}L_{\pm 1}^{(n)}L_0^{(m-i)},\\
    L_1^{(m)}L_{-1}^{(n)} &= \sum_{i=0}^{\min\{m,n\}}\sum_{j=0}^i \binom{-m-n+2i}{j}(-1)^i L_{-1}^{(n-i)}L_0^{(i-j)}L_1^{(m-i)},\\
    L_{\pm 1}^{(m)}L_0^{(n)} &= \sum_{i=0}^n\binom{\mp 2m}{i}L_0^{(n-i)}L_{\pm 1}^{(m)}.
\end{align*}
The last identity is redundant. We also define $\mathcal{H} = U(\mathfrak{sl}_2)_\Z\otimes_\Z\kk$.
We define the generating series
\[e^{zL_{\pm 1}} = \sum_{k\ge 0}z^kL_{\pm 1}^{(k)},\qquad (1+z)^{-2L_0} = \sum_{k\ge 0}z^kL_0^{(k)}.\]
It is helpful to write the above identities in terms of these generating series. The above identities amount to the following:
\begin{align*}
    (1+z)^{-2L_0}(1+w)^{-2L_0} &= (1+z+w+zw)^{-2L_0},\\
    e^{zL_{\pm 1}}e^{wL_{\pm 1}} &= e^{(z+w)L_{\pm 1}},\\
    (1+z)^{-2L_0}e^{wL_{\pm 1}} &= e^{(1+z)^{\mp 1}wL_{\pm 1}}(1+z)^{-2L_0},\\
    e^{xL_1}e^{zL_{-1}} &= e^{(1-xz)^{-1}zL_{-1}}(1-xz)^{-2L_0}e^{(1-xz)^{-1}xL_1}.
\end{align*}
The algebra $\mathcal{H}$ also has the structure of a bialgebra. For $k\in \{-1,0,1\}$ and $n\in \N$, we set
\begin{equation*}
    \varepsilon(L_{k}^{(n)}) = \delta_{n,0},\qquad
    \Delta(L_{k}^{(n)}) = \sum_{i=0}^nL_{k}^{(n-i)}\otimes L_{k}^{(i)}.
\end{equation*}
Furthermore, there is a natural anti-automorphism $\theta\colon \mathcal{H}\to \mathcal{H}^\op$ given by
\[\theta(L_{\pm 1}^{(n)}) = L_{\mp 1}^{(n)},\quad \theta(L_0^{(n)}) = L_0^{(n)}\quad \forall n\in \N.\]
There is also a natural $\Z$-grading on $\mathcal{H}$ given by $\deg(L_{\pm 1}^{(n)}) = \mp n$ and $\deg(L_0^{(n)}) = 0$ for all $n\in \N$.

Now we study modules for $\mathcal{H}$. While all modules we consider are admissible (i.e., $\N$-graded), it may not be the case that $L_0^{(n)}$ respects the $\N$-grading in the obvious sense. Thus, it is better to work with modules equipped with a $\Lambda$-grading, where again $\Lambda$ is a subset of $\Q$ that is closed under the action of $(\Z,+)$. For example, $\Lambda$ could be $c_W+\Z\subset \Q$, so $W$ has conformal dimension $c_W\in \Q$.
\begin{definition}
    A \textit{$\Lambda$-graded weight $\mathcal{H}$-module} is a $\Lambda$-graded $\mathcal{H}$-module $W = \bigoplus_{\ell\in \Lambda}W_\ell$ such that $L_0^{(n)}$ acts as $\binom{-2\deg}{n}$ for all $n\in \N$ in a well-defined manner. That is, for all $\ell\in \Lambda$ such that $W_\ell\neq 0$, the denominator of $\ell$ is invertible in $\kk$ and
    \[L_0^{(n)}|_{W_\ell} = \binom{-2\ell}{n}.\]
    If $W_\ell = 0$, then we set $L_0^{(n)}\cdot W_\ell = 0$.
\end{definition}
\begin{definition}
    A \textit{M\"obius vertex $\kk$-algebra} (or \textit{$\mathcal{H}$-module vertex $\kk$-algebra}) is a $\Z$-graded vertex $\kk$-algebra $V$ with the structure of a $\Z$-graded weight $\mathcal{H}$-module such that:
    \begin{enumerate}
        \item $V_n=0$ for $n\ll 0$.
        \item $L_1^{(n)}\vac = \varepsilon(L_1^{(n)})\vac = \delta_{n,0}\vac$.
        \item For $a\in V$, we have
        \[e^{zL_1}Y(a,z_0)e^{-zL_1} = Y\left(e^{z(1-zz_0)L_1}(1-zz_0)^{-2L_0}a,\frac{z_0}{1-zz_0}\right).\]
    \end{enumerate}
\end{definition}
We can use these structures to define contragredient modules. While all modules we consider are admissible, it may not be the case that $L_0^{(n)}$ respects the grading. Thus, it is better to work with modules equipped with a $\Lambda$-grading.
\begin{definition}\label{contragredient}
    Let $V = \bigoplus_{r\ge K}V_{r}$ be a M\"obius vertex $\kk$-algebra. Given a $\Lambda$-graded $V$-module $W = \bigoplus_{\ell\in \Lambda} W_{\ell}$ whose grading is lower-truncated, we define its \textit{contragredient module} to be the graded dual $W' = \bigoplus_{\ell\in \Lambda} W_{\ell}^\lor$, where $W_{\ell}^\lor = \Hom_{\kk}(W_{\ell},\kk)$, and
    \[Y^{W'}(-,z)\colon V\to \End(W')[\![z,z^{-1}]\!]\]
    is the linear map determined by the relationship
    \[\left\langle Y^{W'}(a,z)\psi,w\right\rangle = \left\langle \psi,Y^{W}(e^{zL_1}(-z^{-2})^{L_0}a,z^{-1})w\right\rangle\]
    for $a\in V$, $\psi\in W'$, $w\in W$, where $\langle \cdot,\cdot\rangle$ is the natural dual pairing. If $a\in V$ is homogeneous, then this identity reads as
    \[\langle a^{W'}_{(n)}\psi,w\rangle = (-1)^{\deg(a)}\sum_{i\ge 0}\langle\psi,(L_1^{(i)}a)^W_{(2\deg(a)-n-i-2)}w\rangle.\]
\end{definition}
Since $L_1^{(i)}$ is a degree $-i$ operator and the $\Z$-grading on $V$ is assumed to be lower-truncated, for $a\in V$ and $i \gg 0$ we have $L_1^{(i)}a = 0$. This shows that the above expression is well-defined. We also have the following results from \cite{FHL} as their proof methods work identically.
\begin{lemma}[\cite[Proposition 3.8]{Li18}, \cite[Theorem 5.2.1]{FHL}]
    The above data defines a $\Lambda$-graded $V$-module $W'$ over $\kk$ whose grading is lower-truncated.
\end{lemma}
\begin{proposition}[\cite[Proposition 5.3.1]{FHL}]
    Given an admissible $V$-module $W$, there is a natural map of $V$-modules $W\to W''$. In particular, $W$ is quasi-reflexive if it is quasi-rigid.
\end{proposition}
The results we cite in the above statements are not stated in our generality, but the same proof arguments work.

Now we consider $V$-modules with a compatible $\mathcal{H}$-module structure.
\begin{definition}
    Let $V$ be a M\"obius vertex $\kk$-algebra. A \textit{$\Lambda$-graded $(V,\mathcal{H})$-module} is a $\Lambda$-graded weight $\mathcal{H}$-module $W = \bigoplus_{\ell\in \Lambda}W_\ell$ with the structure of a $\Lambda$-graded $V$-module such that the following hold:
    \begin{align*}
        e^{zL_{-1}}Y^W(a,x)e^{-zL_{-1}} &= Y^W(e^{zL_{-1}}a,x)\\
        e^{zL_1}Y^W(a,z_0)e^{-zL_1} &= Y^W\left(e^{z(1-zz_0)L_1}(1-zz_0)^{-2L_0}a,\frac{z_0}{1-zz_0}\right).
    \end{align*}
\end{definition}
From the definition, a M\"obius vertex $\kk$-algebra $V$ is naturally a $\Z$-graded $(V,\mathcal{H})$-module.

\subsection{Involutions}
Given a M\"obius vertex $\kk$-algebra $V$, we have a linear map
\[\theta\colon V\otimes_{\kk}\kk[t,t^{-1}]\to V\otimes_{\kk}\kk[t,t^{-1}],\]
which is given for homogeneous elements $a\in V$ by
\[a\otimes \sum_{i}c_it^i \mapsto (-1)^{\deg(a)}\sum_{j\ge 0} L_1^{(j)}a\otimes \sum_{i}c_i t^{2\deg(a)-i-j-2},\]
and extended $\kk$-linearly. The above infinite sum is actually a finite sum since $L_1^{(j)}$ is a degree $-j$ operator, and the grading on $V$ is assumed to be lower-truncated.
The map $\theta$ is related to the involution $\gamma = e^{L_1}(-1)^{L_0}\colon V\to V$ defined on homogeneous elements $a\in V$ by
\[a\mapsto (-1)^{\deg(a)}\sum_{k\ge 0}L_1^{(k)}a,\]
again extended linearly. To state the relationship, we define for $a\in V$ homogeneous
\[J_n(a) = a_{[\deg(a)-1+n]} = \overline{a\otimes t^{\deg(a)-1+n}}.\]
\begin{lemma}[\cite[Lemma 3.4.2]{DGK23}]
    For $a\in V$ homogeneous, we have $\theta(J_n(a)) = J_{-n}(\gamma(a))$. Equivalently,
    \[\theta(a_{[j]}) = (-1)^{\deg(a)}\sum_{i\ge 0}(L_1^{(i)}a)_{[2\deg(a)-j-i-2]}.\]
\end{lemma}
\begin{lemma}[\cite[Lemma 3.4.3]{DGK23}]
    The map $\theta$ defines a Lie algebra involution on $\LVf$ such that
    \begin{align*}
        \theta(\LVf_{d}) &= \LVf_{-d}.
    \end{align*}
\end{lemma}
\begin{lemma}[\cite[Lemma 3.4.4]{DGK23}]
    The Lie algebra involution $\theta\colon \LVf\to \LVf$ induces an $\kk$-algebra involution $U(\LVf)\to U(\LVf)^\op$, which extends to an involution of the universal enveloping algebra $\UV$ of $V$.
\end{lemma}
This involution is precisely what allows us to identify contragredient duals as left $V$-modules instead of right $V$-modules.

\section{Zhu theory}\label{sec:Zhu theory}
Now that we have defined the universal enveloping algebra, we can define higher Zhu algebras over arbitrary rings. We can characterize rationality and many other properties of vertex algebras using this theory.

\subsection{The Zhu algebras}
There are two standard ways of defining Zhu algebras, namely as a quotient of $V$ or as a quotient of $\UV_0$. We prefer the latter definition, though we will have to work with the former at times.
We recommend \cite{DLM2}, \cite{DGK23} and \cite{HeHigherZhu} for Zhu algebra theory over $\C$. We also recommend \cite{Ren17} for the theory of Zhu algebras over an algebraically closed field of characteristic $p>2$.

\begin{definition}\label{Zhu algs}
Let $V$ be a $\Z$-graded vertex $\kk$-algebra.
The $d$-th \textit{(higher) Zhu algebra} is the quotient $\kk$-algebra
\begin{equation}
    \Aa_d(V) = \UV_0\big/\NLR[d+1]\UV_0.
\end{equation}
If $V$ is clear in context, we will just write $\Aa_d$ instead of $\Aa_d(V)$. We also write $\Aa_0 = \Aa$, and we refer to the zeroth Zhu algebra as just the \textit{Zhu algebra}. There are natural surjections $\Aa_d\twoheadrightarrow \Aa_{d-1}$ for all $d\ge 1$.
\end{definition}
This is \textit{not} the ordinary definition of the Zhu algebras. Usually the Zhu algebras are defined as a quotient of $V$. We will discuss this in \zcref{sec equiv Zhu}.
\begin{convention}
    Given a noncommutative ring $A$, we sometimes refer to left $A$-modules as just $A$-modules.
\end{convention}
\begin{convention}
    We will frequently be concerned with the semisimplicity of the Zhu algebras. We say that an associative, unital $\kk$-algebra $A$ is \textit{semisimple} if its module category is semisimple. This is equivalent to $A$ being left/right Artinian and having vanishing Jacobson radical (intersection of maximal left ideals).
\end{convention}
\begin{lemma}
    Let $V$ be a $\Z$-graded vertex $\kk$-algebra, and let $M$ be an admissible $V$-module. Then the degree $d$ part $M_d$ canonically has the structure of an $\Aa_d$-module via the action of $\UV_0$.
\end{lemma}
\begin{proof}
    This comes from the fact that $M$ is equivalently a continuous $\N$-graded module. The image of the action of $\NLR[d+1]\UV_0 = \UV\UV_{\le -d-1}\cap \UV_0$ is then zero on $M_d$.
\end{proof}
There are several functors relating $V$-modules and modules for the above algebras.
\begin{deflem}
    Given an $\Aa_d$-module $\mathsf{M}$, we define the \textit{induced Verma module} $\PhiL_d(\mathsf{M})$ to be
    \[\PhiL_d(\mathsf{M}) = \UV/\NL[d+1]\UV \otimes_{\Aa_d(V)}\mathsf{M}.\]
    We assign the grading 
    \[\PhiL_d(\mathsf{M})_{m} = \UV_{m-d}/\NL[d+1]\UV_{m-d} \otimes_{\Aa_d(V)}\mathsf{M},\]
    and it is naturally a continuous left $\UV$-module. We have $\PhiL_d(\mathsf{M})_{m} = 0$ for $m < 0$.
    Given a graded, continuous left $\UV$-module $W$, we define
    \[\OmegaL_d(W) = \{w\in W \mid (\NL[d+1]\UV)\cdot w = 0\}.\]
    There is an adjunction $\PhiL_d\dashv\OmegaL_d$ between the categories of graded, continuous left $\UV$-modules and left $\Aa_d$-modules \cite[Proposition B.1.4]{DGK23}.
\end{deflem}
Similar to the notation $\Aa_0 = \Aa$, we will often write $\PhiL_0 = \PhiL$. There is a right-handed version of the above constructions, where we replace $\mathsf{L}$ everywhere with $\mathsf{R}$.

It is possible that $\PhiL_d(\mathsf{M})_{0} = 0$, in which case it can be shown that this implies the action of $\Aa_d$ on $\mathsf{M}$ factors through $\Aa_{d-1}$ via the surjection $\Aa_d\twoheadrightarrow \Aa_{d-1}$.
Note also that $\PhiL_d(\mathsf{M})$ is generated by the $\UV$-action on the degree $d$ part $\PhiL_d(\mathsf{M})_d \cong \mathsf{M}$.

\begin{construction}
    It may be the case that a graded, continuous left $\UV$-module $W$ is not generated in degree 0 by the $\UV$-action. We define $J_W\subset W$ to be the maximal graded $\UV$-submodule of $W$ such that the degree $d$ part is $J_W\cap W_d = 0$. We denote $J_W$ by $J$ when the module $W$ is clear in context. Given an $\Aa_d(V)$-module $\mathsf{M}$, we define the \textit{reduced induced Verma module}
    \[\ThetaL_d(\mathsf{M}) = \PhiL_d(\mathsf{M})/J.\]
    By definition we have $\ThetaL_d(\mathsf{M})_d\cong \PhiL_d(\mathsf{M})_d \cong \mathsf{M}$, and $\ThetaL_d(\mathsf{M})$ is generated in degree $d$ by its $\UV$-action.
\end{construction}
Here are some useful facts about the functors $\PhiL_d$ and $\ThetaL_d$:
\begin{lemma}\label{ThetaL simple}
    Let $\mathsf{M}$ be a nonzero $\Aa_d$-module. 
    \begin{enumerate}
        \item Any nonzero graded submodule of $\ThetaL_d(\mathsf{M})$ has nonzero degree $d$ part.
        \item $\ThetaL_d(\mathsf{M})$ is a simple graded $\UV$-module if and only if $\mathsf{M}$ is a simple $\Aa_d(V)$-module.
        \item Suppose $\mathsf{M}$ is simple. Then $\PhiL_d(\mathsf{M})$ is a simple graded left $\UV$-module if and only if $\PhiL_d(\mathsf{M}) \cong \ThetaL_d(\mathsf{M})$ (that is, $J = 0$).
    \end{enumerate}
\end{lemma}
\begin{proof}
    Let $W\subset \ThetaL_d(\mathsf{M})$ be a nonzero graded submodule. Given any homogeneous element $w+J\in W$ with nonzero degree, then $w\notin J$ implies that there exists $u\in \UV_{d-\deg(w)}$ such that $uw \neq 0$. This shows the first statement.

    For the second statement, suppose that $\mathsf{M}$ is simple. Then every nonzero graded submodule $W\subset \ThetaL_d(\mathsf{M})$ gives a $\UV_0$-submodule $W_d\subset \ThetaL_d(\mathsf{M})_d \cong \mathsf{M}$. However, the $\UV_0$-action on $\mathsf{M}$ descends to an $\Aa_d$-action, so $W_d$ is an $\Aa_d$-submodule of $\mathsf{M}$. Therefore $W_d = \mathsf{M}$, which implies that $W = \ThetaL_d(\mathsf{M})$ since the action of $\UV$ on the degree $d$ part $\ThetaL_d(\mathsf{M})_d \cong \mathsf{M}$ generates the whole module.
    Conversely, suppose that $\ThetaL_d(\mathsf{M})$ is simple. If we have a nonzero graded submodule $W_d\subseteq \mathsf{M}$, then we define $W\subseteq \ThetaL_d(\mathsf{M})$ to be the submodule generated by the $\UV$-action on $W_d$. But we must have $W = \ThetaL_d(\mathsf{M})$ by simplicity, so $W_d = \mathsf{M}$. This shows the second statement.

    For the third statement, note that if $\PhiL_d(\mathsf{M})$ is simple, then $\mathsf{M}$ is a simple $\Aa_d$-module by the same arguments as above. This implies $\ThetaL_d(\mathsf{M})$ is simple. Thus we have a surjective map $\PhiL_d(\mathsf{M})\twoheadrightarrow \ThetaL_d(\mathsf{M})$ between two nonzero simple modules, so it must be an isomorphism. The other direction follows from the previous statement.
\end{proof}
Here is a sort of generalization of (b) from \zcref{ThetaL simple} for $\PhiL_d$:
\begin{lemma}\label{indecomposable}
    Let $\mathsf{M}$ be an indecomposable $\Aa_d$-module. Then $\PhiL_d(\mathsf{M})$ is an indecomposable admissible $V$-module.
\end{lemma}
\begin{proof}
    Suppose $\PhiL_d(\mathsf{M})$ can be written as a direct sum $W^1\oplus W^2$. Then there is a decomposition $\mathsf{M} \cong \PhiL_d(\mathsf{M})_d = W^1_d\oplus W^2_d$ as $\Aa_d$-modules. Therefore, one of the two summands is 0. Suppose that $W^2_d = 0$; then
    \[\PhiL_d(\mathsf{M}) = \UV\cdot \mathsf{M} = \UV\cdot W^1_d = W^1.\]
    This implies that $W^2=0$.
\end{proof}



\subsection{An equivalent presentation of the Zhu algebras}\label{sec equiv Zhu}
Our definition of the Zhu algebras is not standard in the literature. Indeed, the original definition of $\Aa_d(V)$ is as a quotient of $V$ \cite{DLM2}. We will show these definitions are equivalent for vertex algebras over an arbitrary ring $\kk$, following the proof of \cite{HeHigherZhu}.

First, we give some setup. Let $V$ be a $\Z$-graded vertex $\kk$-algebra.
For $a,b\in V$ and $m,n,\ell\in \Z$, define
\begin{align*}
    {}^1J^{a,b}_{m,n,\ell} &= \sum_{i\ge 0}\binom{m}{i}(a_{(\ell+i)}b)_{(m+n-i)}\\
    {}^2J^{a,b}_{m,n,\ell} &= \sum_{i\ge 0}(-1)^i\binom{\ell}{i}\left(a_{(m+\ell-i)}b_{(n+i)} - (-1)^\ell b_{(n+\ell-i)}a_{(m+i)}\right).
\end{align*}
For $a\in V$ homogeneous, we set $J_n(a) = a_{(\deg(a)-1+n)}$. Then $J_n(a)$ is a degree $-n$ operator. We may also define $J_n$ on all of $V$ using $\kk$-linearity.

We define the following expressions:
\begin{align*}
    {}^{(1)}J^{a,b}_{m,n,\ell} 
    &= {}^{1}J^{a,b}_{m+\deg(a)-1,n+\deg(b)-1,\ell} \\
    &= \sum_{i\ge 0}\binom{m+\deg(a)-1}{i}J_{m+n+\ell}(a_{(\ell+i)}b)\\
    {}^{(2)}J^{a,b}_{m,n,\ell} 
    &= {}^{2}J^{a,b}_{m+\deg(a)-1,n+\deg(b)-1,\ell}\\
    &= \sum_{i\ge 0}(-1)^i\binom{\ell}{i}\left(J_{m+\ell-i}(a)J_{n+i}(b) - (-1)^\ell J_{n+\ell-i}(b)J_{m+i}(a)\right).
\end{align*}
Then the Jacobi identity is equivalent to ${}^1J^{a,b}_{m,n,\ell} = {}^2J^{a,b}_{m,n,\ell}$ for all $m,n,\ell\in \Z$. Here we prefer to use ${}^{(1)}J^{a,b}_{m,n,\ell}$ and ${}^{(2)}J^{a,b}_{m,n,\ell}$ because it is easier to read the degrees of the operators in the above expressions.

Now, we define some bilinear products from \cite{Ren17}.
Given nonnegative integers $s\le t$ and $d\ge 0$, we define
\begin{equation}
    a\circ^s_{d,t}b = \Res_z \left(Y(a,z)b\frac{(1+z)^{\deg(a) + d+s}}{z^{2d+2+t}}\right) = \sum_{j\ge 0}\binom{\deg(a)+d+s}{j}a_{(j-2d-2-t)}b.
\end{equation}
By Vandermonde's identity, we have
\begin{align*}
    a\circ^s_{d,t}b 
    &= \sum_{i,j\ge 0}\binom{\deg(a)+d}{i}\binom{s}{j}a_{(j+i-2d-2-t)}b\\
    &= \sum_{j\ge 0}\binom{s}{j}a\circ^0_{d,t-j}b.
\end{align*}
For $d\ge 0$, we define $O_d(V)\subset V$ to be the $\kk$-submodule generated by elements of the form $a\circ^0_{d,t}b$ and $T^{(k)}a - \binom{-\deg(a)}{k}a$ for all homogeneous $a,b\in V$ and nonnegative integers $k,t\ge 0$. By the above equation, it would be redundant to include the products $a\circ^s_{d,t}b$. We define the quotient 
\begin{equation}
    A_d(V) = V/O_d(V).
\end{equation}
Note that we are using $A$ and not $\Aa$.
We endow $A_d(V)$ with the product induced by $a*_db$, which is defined on homogeneous $a,b\in V$ as
\begin{align*}
    a*_db &= \sum_{m=0}^d (-1)^m\binom{m+d}{d}\Res_z\left(Y(a,z)b\frac{(1+z)^{\deg(a)+d}}{z^{d+m+1}}\right)\\
    &= \sum_{m=0}^d (-1)^m\binom{m+d}{d}\sum_{i\ge 0}\binom{\deg(a)+d}{i}a_{(i-m-d-1)}b.
\end{align*}
The same arguments as in \cite[Theorem 3.2(1)]{Ren17} show that $A_d(V)$ is indeed an associative unital $\kk$-algebra, where the unit is the class $[\vac]$.

We will show that $\Aa_d(V) \cong A_d(V)$, but to do so we must recall some of the representation theory of $A_d(V)$.
\begin{deflem}
    Let $W$ be a weak $V$-module, and let $d\ge 0$. We define
    \[\Omega_d(W) = \{w\in W \mid a_{(i)}^Ww = 0,\ a\in V,\ i\ge \deg(a)+d\}.\]
    Then $\Omega_d(W)$ is an $A_d(V)$-module such that $[a]$ acts as $o(a) = a_{(\deg(a)-1)}$ for homogeneous $a\in V$.
\end{deflem}
\begin{proof}
    The proof follows from the same arguments as in \cite[Theorem 3.2(3)]{Ren17}.
\end{proof}
One may check that $\Omega_d$ agrees with $\OmegaL_d$ defined above for $\Aa_d(V)$-modules.
\begin{lemma}[\cite[Theorem 3.2(4)]{Ren17}]
    If $M = \bigoplus_{i\ge 0}M_i$ is an admissible $V$-module, then $M_i$ is an $A_d(V)$-submodule of $\Omega_d(M)$ for $i=0,\ldots,d$.
\end{lemma}
It can also be shown that $\Omega_d$ admits a left adjoint functor $M_d$, which assigns to an $A_d(V)$-module $U$ an admissible $V$-module $M_d(U)$. Most importantly, given an $A_d(V)$-module $N$, the set $M_d(N) = \bigoplus_{i\ge 0}M_d(N)_i$ is naturally an admissible $V$-module such that $M_d(N)_d \cong N$. It is possible that $M_d(N)_0 = 0$ if $d>0$, though we will not discuss this possibility since it is not relevant.

Now, we prove that the above definitions are equivalent to the first definition of the Zhu algebras we gave, following \cite{HeHigherZhu}. For this, we need a few results from their work.
\begin{lemma}[\cite[Corollary A.2]{HeHigherZhu}]
    In the universal enveloping algebra $\UV$, for any integers $s,t$ and $N$ satisfying $N+s\ge 0$, we have the identity
    \begin{align*}
        &J_{-s}(a)J_{t}(b)\\
        & =\sum_{j=0}^N\sum_{i\ge 0}(-1)^i\binom{N+\deg(a)}{i}\binom{-N-s-1}{j}J_{t-s}(a_{(-N-s-i-j-1)}b)\\
        &\quad -\sum_{k\ge N+1}\sum_{j=0}^N(-1)^j\binom{N+s+j}{j}\binom{N+s-k}{k-j}J_{-k-s}(a)J_{k+t}(b)\\
        &\quad +\sum_{j=0}^N\sum_{i\ge 0}(-1)^{N+s+1}\binom{N+s+j}{j}\binom{N+s+j+i}{i}J_{t-N-s-1-i}(b)J_{N+1+i}(a).
    \end{align*}
\end{lemma}
\begin{proof}
    This identity follows from \cite[Lemma A.1]{HeHigherZhu}, which holds over any ring $\kk$.
\end{proof}
\begin{lemma}\label{lem:Lem3.1He}
Every element $\sum J_{n_1}(a_1)\cdots J_{n_m}(a_m)$ in $\Aa_n(V)$ can be expressed as $J_0(u)$ for some $u\in V$.
\end{lemma}
\begin{proof}
    The proof is exactly the same as in \cite[Lemma 3.1]{HeHigherZhu}, using \cite[Corollary A.2]{HeHigherZhu}. Note that we do not need to assume $V$ is of CFT-type or even $\N$-graded. It suffices to assume $V$ is $\Z$-graded.
\end{proof}
Now we are ready to state the main result here.
\begin{proposition}\label{Zhu alg iso}
    Let $V$ be a $\Z$-graded vertex $\kk$-algebra. Then the map $V\to \UV_0$ given by $a\mapsto J_0(a)$ induces an isomorphism $A_d(V)\cong \Aa_d(V)$.
\end{proposition}
\begin{proof}
    Consider the map $\varphi\colon V\to\UV_0$ given by $a\mapsto J_0(a)$. By \zcref{lem:Lem3.1He}, we have that the map $\varphi$ induces a surjective morphism $V\to \Aa_d(V)$.
    
    We analyze the map $\varphi$.
    \begin{itemize}
        \item First, we have for all homogeneous $a\in V$ and $k\ge 0$ that
        \[J_0(T^{(k)}a) = \binom{-\deg(a)}{k}J_0(a).\]
        This can be shown from the associator formula since $T^{(k)}a = a_{(-1-k)}\vac$.
        \item Second, note that $J_0(a\circ^s_{d,t}b) = {}^{(1)}J^{a,b}_{d+1+s,d+1+t-s,-2d-2-t}$ for all $t\ge s\ge 0$ and $d\ge 0$. 
        Reindexing, we have that $J_0(a\circ^{m-d-1}_{d,m+n-2(d+1)}b) = {}^{(1)}J^{a,b}_{m,n,-m-n}$ for all $m,n\ge d+1\ge 1$. 

        We consider when ${}^{(1)}J^{a,b}_{m,n,-m-n}\in \NLR[d+1]\UV_0$. By the Jacobi identity, we have ${}^{(1)}J^{a,b}_{m,n,-m-n} = {}^{(2)}J^{a,b}_{m,n,-m-n}$; that is,
        \begin{align*}
            &\sum_{i\ge 0}\binom{m+\deg(a)-1}{i}J_0(a_{(-m-n+i)}b)\\
            &=\sum_{i\ge 0}(-1)^i\binom{-m-n}{i}\left(J_{-n-i}(a)J_{n+i}(b)-(-1)^{m+n}J_{-m-i}(b)J_{m+i}(a)\right).
        \end{align*}
        We have the symmetry ${}^{(1)}J^{a,b}_{m,n,-m-n} = {}^{(1)}J^{b,a}_{n,m,-m-n}$. Note that $J_{m+i}(a)$ has degree $\le -d-1$ whenever $i\ge d+1-m$. Modulo $\NLR[d+1]\UV_0$, we have
        \begin{align*}
            &{}^{(2)}J^{a,b}_{m,n,-m-n}\\
            &\equiv \sum_{i = 0}^{d-\max\{m,n\}}(-1)^i\binom{-m-n}{i}\left(J_{-n-i}(a)J_{n+i}(b)-(-1)^{m+n}J_{-m-i}(b)J_{m+i}(a)\right).
        \end{align*}
        Therefore, ${}^{(1)}J^{a,b}_{m,n,-m-n}\in \NLR[d+1]\UV_0$ for all $n,m\ge d+1\ge 1$ since the right-hand side sum is $0$ then. We conclude that $\varphi(O_d(V)) = 0$.
        \item Third, note that $\varphi(a*_db) \equiv \varphi(a)\varphi(b)$ mod $\NLR[d+1]\UV_0$ by the same argument as in \cite[Theorem 3.2]{HeHigherZhu}.
    \end{itemize}
    From the above, we conclude that the morphism $\varphi\colon V\to \UV_0$ induces a surjective $\kk$-algebra morphism $\varphi\colon A_d(V)\to \Aa_d(V)$.
    
    To see that $\varphi\colon A_d(V)\to \Aa_d(V)$ is injective, take $a\in V$ such that $J_0(a)\in \NLR[d+1]\UV_0$. Then by definition $J_0(a)$ acts on $\Omega_d(M_d(A_d(V)))$ trivially. In particular, $J_0(a)$ acts on the submodule $A_d(V)\subset \Omega_d(M_d(A_d(V)))$ trivially, where the action is given by $J_0(a)\cdot v = a*_dv \equiv 0$ in $A_d(V)$. Setting $v=\vac$, we have $a \equiv a*_d\vac\equiv 0$ in $A_d(V)$, so $a\in O_d(V)$.
\end{proof}
We find that it is easier to work with and calculate $\Aa_d(V)$, but some properties (such as being projective or torsion-free) can be easier to show when working with $A_d(V)$.

\subsection{Mode transition algebras}
Now we consider some objects related to the Zhu algebras. 
\begin{definition}
    Let $n,m\ge 0$. We define the \textit{Zhu bimodule} $\Aa_{n,m}(V)$ to be
    \[\Aa_{n,m}(V) = \UV_{n-m}/\NL[m+1]\UV_{n-m}.\]
    We say that $V$ is \textit{quasi-finite} if $\Aa_{n,m}(V)$ is rigid over $\kk$ for all $n,m\ge 0$.
\end{definition}
Note that $\Aa_d(V) = \Aa_{d,d}(V)$ for all $d\ge 0$.
\begin{lemma}
    For all $n,m,\ell\ge 0$, there is a natural morphism of $(\Aa_{n},\Aa_{\ell})$-bimodules
    \[\mu_{n,\ell}^m \colon \Aa_{n,m}\otimes_{\Aa_m}\Aa_{m,\ell} \to \Aa_{n,\ell}\]
    given by multiplication in $\UV$.
\end{lemma}
The maps above give $\Aa_{n,m}(V)$ the structure of a $(\Aa_n(V),\Aa_m(V))$-bimodule.
\begin{remark}\label{bimodules remark}
    In \cite{DJmz}, the authors define a family of $(A_n(V),A_m(V))$-bimodules, denoted $A_{n,m}(V)$. These are defined as quotients of $V$ instead of a quotient of some graded part of the universal enveloping algebra $\UV$. It is shown in \cite{XuBimodules} that for a VOA $V$ over $\C$, we have an isomorphism $\Aa_{n,m}(V) \cong A_{n,m}(V)$ for all $n,m\ge 0$.\footnote{Actually, they show this isomorphism for the versions of both objects for twisted representations of VOAs. We are not concerned with those in this paper, though.}
    A similar result was shown for vertex operator superalgebras in \cite{Xu2026}. We expect that all of their arguments carry over to an arbitrary ring with minimal changes, though we do not need these results.
\end{remark}
The next class of objects we consider is the mode transition algebras, which are the main object of study in \cite{DGK23} and \cite{DGK24}.
\begin{definition}
    We define the \textit{mode transition algebra}
    \[\Ac(V) = \bigoplus_{n,m\ge 0}\Ac_{n,-m}(V) = \bigoplus_{n,m\ge 0}\Aa_{n,0}\otimes_{\Aa_0}\Aa_{0,m}.\]
    We denote $\Ac_n \eqdef \Ac_{n,-n}$ for all $n\ge 0$.
\end{definition}
In principle, one could study the tensor products $\Aa_{n,d}\otimes_{\Aa_d}\Aa_{d,m}$ for any $d\ge 0$, but for our purposes we only need $d=0$.
\begin{defprop}\label{def:stronglyunital}
    For all $n,m,d\ge 0$, there are natural morphisms of $(\Aa_{n},\Aa_{m})$-bimodules
    \[\star\colon \Ac_{n,-d}\otimes_{\Aa_d}\Ac_{d,-m} \to \Ac_{n,-m},\qquad \mu_{n,m}\colon \Ac_{n,-m} \to \Aa_{n,m}\]
    given by multiplication in $\UV$.
    We refer to the first morphism as the \textit{star product}. We say that $\Ac_n$ is \textit{strongly unital} if there is an element $\one_n \in \Ac_n$, called a \textit{strong identity element}, such that for all $m\ge 0$, $\frak{a} \in \Ac_{m,-n}$, and $\frak{b} \in \Ac_{n,-m}$ we have
    \[\frak{a} \star \one_n = \frak{a},\qquad \one_m\star\frak{b} = \frak{b}.\]
\end{defprop}
\begin{proposition}[\cite[Theorem B.3.3]{DGK23}]\label{Zhu decomposition unital}
    Suppose that $\Ac_d$ admits an identity element (not necessarily strong). Then there is a $\kk$-algebra decomposition
    \[\Aa_d \cong \Ac_d \times \Aa_{d-1}.\]
\end{proposition}
There are other characterizations of strong units given in \cite[Definition/Lemma 3.3.1]{DGK23} and \cite[Lemma 5.1.5]{DGK23}. These results also hold over any ring, but we omit the details.

\subsection{Characterizing rationality}\label{sec:rationality crit}
One of the primary applications of the Zhu algebras and mode transition algebras is in detecting rationality for vertex algebras. For this, we will prove a generalization of the main result of \cite{DGK24}, namely \zcref{DGK rational equiv}. This result restricts to the main result of their work when the base ring is $\C$. Fortunately, all of their arguments generalize in a straightforward manner. Although we will eventually have to require $\kk$ to be a field, most of the intermediate steps can be done over any ring $\kk$.

First, we relate the semisimplicity of the Zhu algebras to the mode transition algebras. The following result is implicit in \cite{DGK24}, written down as part of their larger result \cite[Theorem 4.0.10]{DGK24}.
\begin{proposition}
    Let $V$ be a $\Z$-graded vertex $\kk$-algebra.
    Suppose $\Aa$ is semisimple and $\Ac_d$ is strongly unital for all $d\ge 0$. Then $\Aa_d$ is semisimple for all $d\ge 0$.
\end{proposition}
\begin{proof}
    Since $\Ac_d$ is strongly unital for all $d\ge 0$, by \zcref{Zhu decomposition unital} there is a decomposition
    \begin{equation}\label{eq: Zhu mode transition alg iso}
        \Aa_d\cong \prod_{e=0}^d\Ac_e.
    \end{equation}
    As such, $\Aa_d$ is semisimple for all $d\ge 0$ if and only if $\Ac_d$ is semisimple for all $d\ge 0$.
    
    The entire section \cite[\textsection 3.1]{DGK23} on Peirce algebras is given over an arbitrary base commutative ring, so it applies here. We may then use \cite[Theorem 3.2.1]{DGK23} and conclude that there is an equivalence of categories between $\Ac_d$-modules and $\mathsf{Z}_d$-modules, where $\mathsf{Z}_d$ is a specific two-sided idempotent ideal of $\Aa$ defined in \cite[Definition 3.1.2]{DGK24}. It suffices to show that $\mathsf{Z}_d$ is semisimple for all $d\ge 0$. Since $\Aa$ is semisimple, it follows that the two-sided ideal $\mathsf{Z}_d$ must be semisimple as well.
\end{proof}
Here is an extension of this result:
\begin{corollary}
    Let $V$ be a $\Z$-graded vertex $\kk$-algebra. Suppose that $\Aa$ is semisimple and finitely generated over $\kk$. If $\Ac_d$ is strongly unital for all $d\ge 0$, then $\Aa_d$ is semisimple and finitely generated over $\kk$ for all $d\ge 0$.
\end{corollary}
\begin{proof}
    The semisimplicity of $\Aa_d$ is obvious. Note that $\mathsf{Z}_d$ is finitely generated over $\kk$ since it is an $\kk$-submodule of $\Aa$. Therefore, $\Ac_d$ is finitely generated over $\kk$ because it is Morita equivalent to $\mathsf{Z}_d$. We obtain the desired statement from the formula in \zcref[noname]{eq: Zhu mode transition alg iso}.
\end{proof}

We now discuss some properties of contragredient duals. For this, we let $V$ be a M\"obius vertex $\kk$-algebra with Zhu algebra $\Aa$.
\begin{definition}
    Given an $\Aa$-module $\mathsf{M}$, we can identify ${}^\theta\mathsf{M}^\lor$ with the degree zero parts of $\PhiL({}^\theta\mathsf{M}^\lor)$ and $\PhiL(\mathsf{M})'$. By \cite[Proposition B.1.4]{DGK23}, there is a natural map 
    \[\psi_{\mathsf{M}}^\lor\colon \PhiL({}^\theta\mathsf{M}^\lor)\to \PhiL(\mathsf{M})'.\]
    For a similar reason, there is a natural map 
    \[\psi_{\mathsf{M}}\colon \PhiL(\mathsf{M}) \to \PhiL({}^\theta\mathsf{M}^\lor)'\]
    given by the map $\mathsf{M}\to {}^{\theta\theta}\mathsf{M}^{\lor\lor}$, which factors through $\psi_{{}^\theta\mathsf{M}^\lor}^\lor$.
\end{definition}
\begin{remark}\label{fin gen proj over k module}
    If $\mathsf{M}$ is a reflexive left $\Aa$-module, then there are natural identifications $\psi_{{}^\theta\mathsf{M}^\lor} = \psi_{\mathsf{M}}^\lor$ and $\psi_{{}^\theta\mathsf{M}^\lor}^\lor = \psi_{\mathsf{M}}$.

    Note that if the Zhu algebra $\Aa$ is semisimple over $\kk$, then it may not be the case that every simple $\Aa$-module $\mathsf{M}$ is rigid or even reflexive over $\kk$. However, if $\Aa$ is finitely generated over $\kk$, then every simple $\Aa$-module will be finitely generated over $\kk$.
\end{remark}
\begin{lemma}[\cite[Lemma 4.0.6]{DGK23}]\label{quasi-rigid implies simple}
    Let $\kk$ be a field. Let $\mathsf{M}$ be a simple finite-dimensional $\Aa$-module such that the map $\psi_{\mathsf{M}}$ is an isomorphism and $\PhiL({}^\theta\mathsf{M}^\lor)$ is quasi-rigid. Then the $V$-modules $\PhiL(\mathsf{M})$, $\PhiL({}^\theta\mathsf{M}^\lor)$, and $\PhiL({}^\theta\mathsf{M}^\lor)'$ are all simple.
\end{lemma}
\begin{proof}
    By \zcref{ThetaL simple}, we have short exact sequences
    \[0\to J\to \PhiL(\mathsf{M})\to \ThetaL(\mathsf{M})\to 0,\qquad 0\to \widetilde{J}\to \PhiL({}^\theta\mathsf{M}^\lor)\to \ThetaL({}^\theta\mathsf{M}^\lor)\to 0.\]
    Taking the contragredient dual, we obtain a diagram
    \begin{equation}
        \begin{tikzcd}
            0 \ar[r] & J \ar[r] & \PhiL(\mathsf{M}) \ar[r] \ar{d}{\psi_{\mathsf{M}}} & \ThetaL(\mathsf{M}) \ar[r] & 0\\
            0 \ar[r] & \ThetaL({}^\theta\mathsf{M}^\lor)' \ar[r] & \PhiL({}^\theta\mathsf{M}^\lor)' \ar[r] & \widetilde{J}' \ar[r] & 0.
        \end{tikzcd}
    \end{equation}
    The rows above are exact. By definition, the degree zero parts of $J$, $\widetilde{J}$, $J'$, $\widetilde{J}'$ are 0. Since we have a surjective morphism $\PhiL(\mathsf{M})\to \widetilde{J}'$ such that the degree zero part of $\widetilde{J}'$ is 0, we may conclude that $\widetilde{J}' = 0$. But $\widetilde{J}$ is quasi-rigid since it is a submodule of $\PhiL({}^\theta\mathsf{M}^\lor)$, so we must have that $\widetilde{J}'=0$ implies $\widetilde{J} = 0$. Therefore $\PhiL({}^\theta\mathsf{M}^\lor) \cong \ThetaL({}^\theta\mathsf{M}^\lor)$ is simple. Taking the above diagram and replacing $\mathsf{M}$ with ${}^\theta\mathsf{M}^\lor$, we have that $\PhiL(\mathsf{M}) \cong \PhiL({}^\theta\mathsf{M}^\lor)'$ is simple as well.
\end{proof}

\begin{lemma}\label{psi iso implies quasi-rigid}
    Suppose $\mathsf{M}$ is a simple $\Aa$-module such that the canonical maps $\psi_{\mathsf{M}}$ and $\psi_{{}^\theta\mathsf{M}^\lor}$ are isomorphisms. Then $\PhiL(\mathsf{M})$ and $\PhiL({}^\theta\mathsf{M}^\lor)$ are quasi-reflexive.
\end{lemma}
\begin{proof}
    If $\psi_{\mathsf{M}}$ is an isomorphism, then the degree 0 part of this map says that $\mathsf{M}\cong {}^{\theta\theta}\mathsf{M}^{\lor\lor}$, so $\mathsf{M}$ is reflexive over $\kk$. Then the composition
    \[(\psi_{{}^\theta\mathsf{M}^\lor})'\circ \psi_{\mathsf{M}}\colon \PhiL(\mathsf{M}) \to \PhiL({}^\theta\mathsf{M}^\lor)' \to \PhiL({}^{\theta\theta}\mathsf{M}^{\lor\lor})'' \cong \PhiL(\mathsf{M})''\]
    is an isomorphism. That is, $\PhiL(\mathsf{M})$ is quasi-reflexive. We may also apply the same reasoning to conclude that $\PhiL({}^\theta\mathsf{M}^\lor)$ is quasi-reflexive since $\psi_{{}^\theta\mathsf{M}^\lor}$ and $\psi_{{}^{\theta\theta}\mathsf{M}^{\lor\lor}} = \psi_{\mathsf{M}}$ are isomorphisms.
\end{proof}
\begin{lemma}\label{lem: simple module equiv}
    Let $\kk$ be a field, and let $\mathsf{M}$ be a simple $\Aa$-module. The following are equivalent:
    \begin{enumerate}
        \item $\psi_{\mathsf{M}}$ and $\psi^\lor_{\mathsf{M}}$ are isomorphisms of admissible $V$-modules;
        \item $\PhiL(\mathsf{M})$ and $\PhiL({}^\theta\mathsf{M}^\lor)$ are simple and quasi-rigid.
    \end{enumerate}
\end{lemma}
\begin{proof}
    If $\psi_{\mathsf{M}}$ is an isomorphism, then $\mathsf{M}$ is finite-dimensional since it is reflexive. By \zcref{fin gen proj over k module} we may identify $\psi_{{}^\theta\mathsf{M}^\lor} = \psi_{\mathsf{M}}^\lor$. Therefore, by \zcref{psi iso implies quasi-rigid}, we conclude that $\PhiL(\mathsf{M})$ and $\PhiL({}^\theta\mathsf{M}^\lor)$ are quasi-rigid since we are over a field. The simplicity of $\PhiL(\mathsf{M})$ and $\PhiL({}^\theta\mathsf{M}^\lor)$ follows from \zcref{quasi-rigid implies simple}.
    
    Conversely, we note that $\PhiL(\mathsf{M})$ and $\PhiL({}^\theta\mathsf{M}^\lor)$ being simple and quasi-rigid implies that $\PhiL(\mathsf{M})'$ and $\PhiL({}^\theta\mathsf{M}^\lor)'$ are simple and quasi-rigid. But then $\psi_{\mathsf{M}}$ and $\psi_{\mathsf{M}}^\lor$ are isomorphisms since they are nonzero maps.
\end{proof}

Suppose $\Aa$ is a semisimple associative $\kk$-algebra. Then the Artin-Wedderburn theorem says that there is a finite set of division $\kk$-algebras $\{D_i\}_i$ and $m_i\ge 1$ such that
\[\Aa \cong \prod_i \Mat_{m_i\times m_i}(D_i).\]
In other words, we may write
\[\Aa \cong \prod_i \mathsf{S}^i_0 \otimes_{D_i}(\mathsf{S}^i_0)^\lor,\]
where $\{\mathsf{S}^i_0\}$ is the set of simple $\Aa$-modules, and $\End_\kk(\mathsf{S}^i_0)^\op = D_i$. Note that the tensor is taken over $D_i$. For our purposes, we must set $D_i = \kk$ for all $\kk$, which forces $\kk$ to be a field.
\begin{proposition}
    Let $\kk$ be a field. Let $V$ be a $\Z$-graded vertex $\kk$-algebra such that the Zhu algebra $\Aa$ is a finite direct product of matrix algebras over $\kk$ (hence semisimple). Then the following are equivalent: 
    \begin{enumerate}
        \item For any simple $\Aa$-module $\mathsf{M}$, the natural map $\psi_{\mathsf{M}}^\lor\colon \PhiL({}^\theta\mathsf{M}^\lor)\to \PhiL(\mathsf{M})'$ is an isomorphism.
        \item For any simple $\Aa$-module $\mathsf{M}$, we have that $\PhiL(\mathsf{M})$ is simple and quasi-rigid.
    \end{enumerate}
    If either of the above equivalent conditions hold, then $\Ac_d$ is strongly unital for all $d\ge 0$.
\end{proposition}
\begin{proof}
    Given the assumptions on $\Aa$, every simple $\Aa$-module $\mathsf{M}$ is finite-dimensional. Note that the dual ${}^\theta\mathsf{M}^\lor$ is simple if $\mathsf{M}$ is. By \zcref{fin gen proj over k module}, we have $\psi_{\mathsf{M}} = \psi_{{}^\theta\mathsf{M}^\lor}^\lor$ for any simple $\Aa$-module $\mathsf{M}$. The equivalence of (a) and (b) follows from \zcref{lem: simple module equiv}.

    Now assume the above conditions hold. By assumption, there is a finite set of positive integers $\{m_k\}_k$ such that
    \[\Aa \cong \prod_k \Mat_{m_k\times m_k}(\kk).\]
    In other words, we may write
    \[\Aa \cong \prod_k \mathsf{S}^k_0 \otimes_{\kk}(\mathsf{S}^k_0)^\lor,\]
    where $\{\mathsf{S}^k_0\}$ is the set of simple $\Aa$-modules, and $\End_\kk(\mathsf{S}^k_0)^\op = \kk$. Note that the tensor is taken over $\kk$. Define $S^k = \PhiL(\mathsf{S}^k_0)$, and denote $S^k_d = \PhiL(\mathsf{S}^k_0)_d$. By \cite[Lemma 3.4.5]{DGK23}
    \[\Ac = \Phi(\Aa) \cong \prod_k \PhiL(\mathsf{S}^k_0)\otimes_{\kk}\PhiR((\mathsf{S}^k_0)^\lor).\]
    Now, via the involution of $\UV$ we can naturally identify $\PhiR((\mathsf{S}^k_0)^\lor)_m = \PhiL({}^\theta(\mathsf{S}^k_0)^\lor)_{-m}$. By hypothesis, we have an isomorphism $\PhiL({}^\theta(\mathsf{S}^k_0)^\lor) \stackrel{\sim}{\to}(S^k)'$. Taking the bidegree $(d,-e)$ part, we have an isomorphism
    \[\Ac_{d,-e} \cong \prod_k S^k_d \otimes_{\kk}(S^k_e)^\lor.\]
    Looking at $e=d$, we have
    \[\Ac_{d,-d} \cong \prod_k S^k_d \otimes_{\kk} (S^k_d)^\lor = \prod_k \End_\kk(S^k_d).\]
    We conclude that $\Ac_{d,-d}$ is strongly unital by taking the sum of the identity operators on each factor in the product.
\end{proof}
We suspect there is some generalization of the above result for arbitrary rings and not just fields, but we do not need such a result for our purposes.

Next, we relate the Zhu algebras to rationality. 
\begin{proposition}\label{rational iff An semisimple PhiL simple}
    Let $V$ be a $\Z$-graded vertex algebra over a ring $\kk$. Then $V$ is quasi-rational if and only if $\Aa_d$ is semisimple for all $d\ge 0$ and the functor $\PhiL_d$ sends simple $\Aa_d$-modules to simple admissible $V$-modules.
\end{proposition}
\begin{proof}
    For an admissible $V$-module $W = \bigoplus_{d\ge 0}W_d$ we can see that $W_d$ is an $\Aa_d$-module for all $d\ge 0$. If $V$ is quasi-rational, then $W$ is semisimple, hence $W_d$ is semisimple as an $\Aa_d$-module. Then $\Aa_d = \PhiL_d(\Aa_d)_d$ is clearly semisimple. Recall that in a semisimple category, the indecomposable objects are precisely the simple objects. By \zcref{indecomposable}, the functor $\PhiL_d$ sends simple $\Aa_d$-modules to simple admissible $V$-modules.

    For the converse, let $W = \bigoplus_{d\ge 0}W_d$ be an admissible $V$-module. Define
    \[\widetilde{W} = \bigoplus_{d\ge 0}\UV\cdot W_d.\]
    Each $\UV\cdot W_d$ is a submodule of $W$ generated by the degree $d$ part $W_d$. By assumption, $W_d$ is a semisimple $\Aa_d$-module, so $\PhiL_d(W_d)$ is a semisimple admissible $V$-module. There is a canonical quotient map $\PhiL_d(W_d) \twoheadrightarrow \UV\cdot W_d$, so $\UV\cdot W_d$ is a semisimple admissible $V$-module. Therefore, $\widetilde{W}$ is a semisimple admissible $V$-module. There is a canonical quotient $\widetilde{W} \twoheadrightarrow W$ given by mapping each $\UV\cdot W_d$ into $W$ naturally, so $W$ is a semisimple admissible $V$-module. We conclude that $V$ is quasi-rational.
\end{proof}
Now, we give a generalization of the results from \cite{DLM2} or \cite{Ren17} on characterizing rationality.
\begin{proposition}\label{prop:DLM rationality}
Let $V$ be a $\Z$-graded vertex algebra over a ring $\kk$, and suppose that:
\begin{enumerate}
    \item $\Aa_d$ is semisimple for all $d\ge 0$;
    \item Given any admissible $V$-module $M$ with $M_0 \neq 0$, we have that $M_n \neq 0$ for $n\gg 0$.
\end{enumerate}
Then $V$ is quasi-rational.
\end{proposition}
\begin{proof}
    By \zcref{rational iff An semisimple PhiL simple}, it suffices to show that for any simple $\Aa_d$-module $\mathsf{M}$, the admissible $V$-module $\PhiL_d(\mathsf{M})$ is simple.
    
    Let $M = \bigoplus_{p\ge 0}M_{p}$ be a nonzero admissible $V$-module generated in degree $d$. Let $J = J_M\subset M$ be the maximal graded submodule of $M$ such that $J_{d}=0$. By assumption, there exists $p > 0$ such that $J_{p+d}\neq 0$ and $\ThetaL_{d}(M_{d})_{p+d} = M_{p+d}/J_{p+d} \neq 0$. Since $\Aa_{d+p}$ is semisimple, we have an isomorphism
    \[M_{p+d} \cong J_{p+d} \oplus M_{p+d}/J_{p+d}.\]
    Let $P = \UV\cdot (M_{p+d}/J_{p+d})\subset \ThetaL_d(M_{d})$, which is generated in degree $p+d$.
    By the definition of $J$, any nonzero class $m+J\in \ThetaL_d(M_{d})_{p+d}$ gives rise to a nonzero submodule of $M_{d}$ via multiplication by an element in $\UV_{-p}$, so we have a nonzero submodule $P_{d}\subseteq M_{d}$. If $M_{d}$ is simple, then 
    \[M_{d} = P_{d} = \UV_{-p}\cdot (M_{p+d}/J_{p+d}).\]
    But this implies that $P$ is also generated in degree $d$, so $P = \ThetaL_d(M_{d})$, which is a simple admissible $V$-module.
    
    Applying $\UV_p$, we obtain 
    \[M_{p+d} = \UV_p\cdot M_{d} = \UV_p\cdot M_{d}/J_{d} = M_{p+d}/J_{p+d}.\]
    Therefore $J_{p+d}=0$, which is a contradiction. This implies that $J = 0$ if $M_{d}$ is simple. To conclude, we have that for a simple $\Aa_d$-module $\mathsf{M}$, every admissible $V$-module $M = \bigoplus_{p\ge 0}M_p$ that is generated by the degree $d$ part $M_d = \mathsf{M}$ is simple. Therefore, $M = \PhiL_d(\mathsf{M}) = \ThetaL_d(\mathsf{M})$. This concludes the proof.
\end{proof}

Now, we show how quasi-rationality and rationality are related.
\begin{lemma}\label{lem:rational conditions}
    Let $V$ be a $\Z$-graded vertex algebra over a field $\kk$. If $V$ is quasi-rational and $\Aa_d$ is finite-dimensional for all $d\ge 0$, then it is rational.
\end{lemma}
\begin{proof}
    Given a simple admissible $V$-module $W = \bigoplus_{d\ge 0}W_d$, the graded component $W_d$ is a simple $\Aa_d$-module for all $d\ge 0$. Indeed, if $\mathsf{M}\subseteq W_d$ is an $\Aa_d$-submodule, then $\UV\cdot \mathsf{M}$ is an admissible submodule of $W$. Hence $\mathsf{M} = 0$ or $W_d$. Note that every simple $\Aa_d$-module is finite-dimensional provided that $\Aa_d$ is finite-dimensional, so we are done.
\end{proof}

We conclude this section with the following combined result. The proof follows from all of the above results.
\begin{theorem}\label{DGK rational equiv}
    Let $V$ be a M\"obius vertex algebra over a field $\kk$. Assume that: 
    \begin{itemize}
        \item The Zhu algebra $\Aa$ is a finite product of matrix algebras over $\kk$ (hence semisimple).
        \item Given any admissible $V$-module $M$ with $M_0 \neq 0$, we have that $M_n \neq 0$ for $n\gg 0$.
    \end{itemize}
    Then the following statements are equivalent:
    \begin{enumerate}
        \item For every simple $\Aa$-module $\mathsf{M}$, the induced Verma module $\PhiL(\mathsf{M})$ is simple and quasi-rigid (has finite-dimensional graded parts).
        \item For every simple $\Aa$-module $\mathsf{M}$, the map $\psi_{\mathsf{M}}^{\lor} = \PhiL({}^\theta\mathsf{M}^\lor) \to (\PhiL(\mathsf{M}))'$ is an isomorphism.
        \item $\Ac_d$ is strongly unital for all $d\ge 0$.
        \item $\Aa_d$ is semisimple and finite-dimensional for all $d\ge 0$.
        \item $V$ is rational.
    \end{enumerate}
\end{theorem}
By \cite[Lemma 4.3]{Ren17} and \cite[Theorem 4.11]{DLM2}, the second assumption is redundant for VOAs over a field. We leave the assumption in, but generally one should expect this to already hold.

\subsection{Base change of a vertex algebra}
Let $\kk\to S$ be a flat morphism of rings, and let $V$ be a vertex $\kk$-algebra. Then we may define the \textit{base change} or \textit{induced vertex algebra}
\[V_{S} = V\otimes_{\kk}S,\]
which is naturally a vertex $S$-algebra as follows. The vacuum element is $\vac_{V_S} = \vac_V\otimes 1$, and the state-field correspondence is given by setting $(a\otimes s)_{(n)} = a_{(n)}\otimes m_s$ (here $m_s$ denotes multiplication by $s$) for all $a\in V$, $s\in S$, and $n\in \Z$.
Similarly, if $V$ is a vertex $S$-algebra, then we may naturally interpret it as a vertex $\kk$-algebra via restriction of scalars. These two constructions form an adjoint pair, where base change is left adjoint.
In this section, we will analyze the base change and the properties it preserves.

Some natural constructions are preserved under base change in an obvious way.
\begin{lemma}\label{lem:base change}
    Let $\kk\to S$ be a flat morphism of rings, and let $V$ be a vertex $\kk$-algebra. 
    Then the base change $(-)\otimes_\kk S$ defines a functor from the category of vertex $\kk$-algebras to the category of vertex $S$-algebras. Moreover, we have the following identities:
    \[\frakL(V_S)^{\mathsf{f}} = \LVf\otimes_{\kk}S,\qquad \UV(V_S) \cong \UV(V)\: \widehat{\otimes}_{\kk}\: S,\]
    \[\Aa_{n,m}(V_S) = \Aa_{n,m}(V)\otimes_\kk S,\qquad \Ac_{n,-m}(V_S) = \Ac_{n,-m}(V)\otimes_\kk S,\]
    where ${}_V\PhiL_d$ denotes the functor $\PhiL_d$ for the vertex algebra $V$. Furthermore, if $\mathsf{M}$ is a left $\Aa_d$-module, then
    \[{}_{V_S}\PhiL_d(\mathsf{M}\otimes_R S) \cong {}_{V}\PhiL_d(\mathsf{M})\otimes_R S.\]
\end{lemma}
\begin{proof}
    All of the other identities follow from the first two. The first identity is easy to see as
    \[V_S\otimes_S S[t,t^{-1}] \cong (V\otimes_\kk\kk[t,t^{-1}]) \otimes_\kk S.\]
    By flatness, we get the desired statement.
    For the second identity, we have the following relation between universal enveloping algebras:
    \[U(\frakL(V_S)^{\mathsf{f}}) \cong U(\LVf)\otimes_\kk S.\]
    The canonical seminorms for $U(\frakL(V_S)^{\mathsf{f}})$ and $U(\frakL(V)^{\mathsf{f}})$ are naturally related via base change. That is,
    \[\NL[m]U(\frakL(V_S)^{\mathsf{f}}) = \NL[m]U(\frakL(V)^{\mathsf{f}}) \otimes_R S.\]
    By flatness, we have
    \[U(\frakL(V_S)^{\mathsf{f}})/\NL[m]U(\frakL(V_S)^{\mathsf{f}}) = U(\frakL(V)^{\mathsf{f}})/\NL[m]U(\frakL(V)^{\mathsf{f}}) \otimes_\kk S.\]
    Taking the completion of $U(\frakL(V_S)^{\mathsf{f}})$, we have by definition
    \[\widehat{U}(\frakL(V_S)^{\mathsf{f}}) \cong \widehat{U}(\LVf)\ \widehat{\otimes}_\kk\ S.\]
    All of the relations of a vertex algebra involve only integer coefficients, so the ideal $\widehat{J}_S\subset \widehat{U}(\frakL(V_S)^{\mathsf{f}})$ defining the vertex algebra relations for $V_S$ is naturally equal to $\widehat{J}\widehat{\otimes}_\kk S$, where $\widehat{J}\subset \widehat{U}(\frakL(V)^{\mathsf{f}})$ is the ideal defining the vertex algebra relations for $V$. The rest of the identities follow from similar observations.
\end{proof}
One could prove the remaining properties without using the universal enveloping algebra if one is willing to use an expression for $\Aa_{n,m}(V)$ as a quotient of $V$ (see \zcref{bimodules remark}).

Let $\mathsf{M}$ be an $\Aa(V)$-module. Then $\mathsf{M}_S = \mathsf{M}\otimes_\kk S$ is naturally an $\Aa(V_S)$-module. From the above lemma, we have
\[{}_{V_S}\PhiL(\mathsf{M}_S) = {}_V\PhiL(\mathsf{M})\otimes_\kk S.\]
Let $J_R$ (resp. $J_S$) be the maximal graded submodule of ${}_{V_S}\PhiL(\mathsf{M}_S)$ (resp. ${}_{V}\PhiL(\mathsf{M})$) with degree 0 part equal to 0, and denote $J_\kk$ similarly for ${}_V\PhiL(\mathsf{M})$. Then we have a natural inclusion
\[J_R\otimes_R S\subseteq J_S.\]
This may not be an equality. Even so, we will use this in our analysis of Virasoro VOAs.

We need some additional standard results from commutative algebra.
\begin{lemma}\label{torsion free lemma}
    Let $\mathbb{D}$ be an integral domain, and let $\mathbb{F} = \operatorname{Frac}(\mathbb{D})$.
    \begin{enumerate}
        \item For any $\mathbb{D}$-module $M$, the kernel of the morphism 
        \[M\to M\otimes_{\mathbb{D}}\mathbb{F},\qquad m\mapsto m\otimes 1\]
        is the submodule $T(M)\subset M$ of torsion elements of $M$.
        \item Let $f\colon M\to N$ be a morphism of $\mathbb{D}$-modules. Assume that $\im(f)$ is torsion-free. Then
        \[\ker(f)/T(M) = \ker(f_{\mathbb{F}}) \cap M/T(M),\]
        where $f_{\mathbb{F}} \colon M\otimes_{\mathbb{D}}\mathbb{F} \to N\otimes_{\mathbb{D}}\mathbb{F}$ is the induced map and $M/T(M)$ is interpreted as a submodule of $M\otimes_{\mathbb{D}}\mathbb{F}$.
    \end{enumerate}
    \end{lemma}
\begin{proof}
    Statement (a) is well-known, so we just prove (b). We identify $m+T(M)\in M/T(M)$ with $m\otimes 1\in M\otimes_{\mathbb{D}}\mathbb{F}$.
    Given an element $m\otimes 1 \in \ker(f_{\mathbb{F}}) \cap M/T(M)$, we have
    \[f_{\mathbb{F}}(m\otimes 1) = f(m)\otimes 1 = 0.\]
    Therefore, $f(m)$ is torsion. But $\im(f)$ is torsion-free, so $f(m)=0$.
\end{proof}
\begin{lemma}\label{lem:base change dual}
    Let $R\to S$ be a morphism of rings, and let $M$ be a rigid $R$-module. Then we have an isomorphism
    \[\Hom_R(M,R) \otimes_RS \cong \Hom_S(M\otimes_RS,S).\]
\end{lemma}
We will use these results when studying Virasoro VOAs.

\subsection{$C_2$-cofiniteness}
Another interesting algebra to study for vertex algebras is the algebra $R_V = V/C_2(V)$.
\begin{definition}
    For $n\ge 2$, we define
    \begin{equation}\label{Cn(V)}
        C_n(V) = \Span_{\kk}\{a_{(-k)}b \mid a,b\in V,\ k\ge n\}.
    \end{equation}
    We say that $V$ is \textit{$C_n$-cofinite (over $\kk$)} if the quotient $\kk$-module $V/C_n(V)$ is finitely generated over $\kk$.
\end{definition}
$R_V = V/C_2(V)$ is naturally equipped with the structure of a commutative Poisson algebra by setting $a\cdot b = a_{(-1)}b$ and $\{a,b\} = a_{(0)}b$.
\begin{remark}
    Note that $C_n(V)\subset C_m(V)$ for $n\ge m$, so there is a surjection $V/C_n(V) \twoheadrightarrow V/C_m(V)$. Therefore, if $V$ is $C_n$-cofinite, then it is $C_m$-cofinite.
\end{remark}
It is straightforward to see that the above definition is compatible with base change. 
\begin{lemma}\label{C2 base change}
    If we have a flat morphism of rings $\kk\to S$ and a vertex $\kk$-algebra $V$, then $C_n(V_S) = C_n(V)\otimes_\kk S$. Moreover,
    \[V_S/C_2(V_S) \cong V/C_2(V)\otimes_\kk S.\]
    If $V$ is $C_2$-cofinite over $\kk$, then $V_S$ is $C_2$-cofinite over $S$.
\end{lemma}

Here is a partial way to identify the structure of $R_V$.
\begin{lemma}[\cite[Lemma 2.15]{Li19}]\label{lem:C2 generating set}
    Let $V$ be a vertex $\kk$-algebra, and let $U\subset V$ be a submodule such that
    \[V = \Span_\kk\{u^{1}_{(-n_1)}\cdots u^{r}_{(-n_r)}\vac \mid r\ge 0,\ u^i\in U,\ n_i\ge 1\}.\]
    Then there exists a surjective $\kk$-algebra homomorphism $\pi$ from the symmetric algebra $S(U)$ on $U$ to $V/C_2(V)$ such that $\pi(u) = u+C_2(V)$ for $u\in U$.
\end{lemma}
\begin{proof}
    The proof is the same as in \cite{Li19}.
\end{proof}

The $C_2$-algebra controls a great deal of information about generating sets for a vertex algebra.
\begin{lemma}\label{lem: R_V surjection gr A(V)}
    Let $V$ be an $\N$-graded vertex $\kk$-algebra. Then there is a natural surjection of associative $\kk$-algebras
    \[V/C_2(V) \twoheadrightarrow \gr \Aa(V),\]
    where the filtration on $\Aa(V)$ is induced from the standard filtration on $V$ coming from its grading.
\end{lemma}
A related result in this direction is that if a VOA $V$ over a characteristic 0 field is $C_2$-cofinite, then it is quasi-finite \cite[Theorem 9.2.1]{MNT10}. Their result may be generalized to work over any field using the exact same arguments.
\begin{theorem}\label{lem:C2 cofinite implies quasi finite}
    Let $V$ be an $\N$-graded vertex algebra over a field $\mathbb{F}$. If $V$ is $C_2$-cofinite, then it is quasi-finite.
\end{theorem}
\begin{proof}
    The proof is identical to the proof for a VOA over a field of characteristic 0 from \cite{MNT10}. In fact, everything in Section 7 onward in \cite{MNT10} may be done over an arbitrary field, and $V_{0}$ need not be one-dimensional as this property is not used anywhere.
\end{proof}
If $V$ is an $\N$-graded vertex $\kk$-algebra that is $C_2$-cofinite over $\kk$, then the base change $V_{\mathbb{F}}$ is $C_2$-cofinite for all fields $\mathbb{F}$ admitting a morphism $\kk\to \mathbb{F}$. This implies that $\Aa_{n,m}(V)\otimes_{\kk}\mathbb{F}$ is finite-dimensional over $\mathbb{F}$ for all such $\mathbb{F}$. However, we cannot necessarily say anything about $\Aa_{n,m}(V)$ from this in full generality.

\section{The Virasoro vertex operator algebra}\label{sec:Virasoro}
In this section, we will apply the general theory from the previous section to study VOAs associated with the Virasoro algebra. Our goal is to obtain some understanding of the simple quotients of the universal Virasoro VOA.

\subsection{The universal Virasoro vertex operator algebra}\label{sec:universal Virasoro}
We recall some notation from \cite{DongRen16}, which we will use for this entire section.
Let $\kk$ be a ring.
We define the Virasoro algebra $\Vir_{\kk}$ to be the Lie algebra 
\[\Vir_{\kk} =  \bigoplus_{n\in \Z}\kk L_n\oplus \kk\frac{\mathbf{c}}{2} \]
with bracket relations
\[[L_m,L_n] = (m-n)L_{m+n} + \frac{\mathbf c}{2}\binom{m+1}{3}\delta_{m+n,0},\qquad [\Vir,\mathbf{c}] = 0.\]
Now let $c,h\in \kk$ such that $\frac{c}{2}$ exists in $\kk$. We set
\[V_{\Vir}(c,h)_{\kk} = U(\Vir_{\kk})\otimes_{U(\Vir_{\kk}^{\ge 0})}\kk 1,\]
where $\Vir_{\kk}^{\ge 0}$ is the subalgebra generated by $L_n$ for $n\ge 0$ and $\mathbf{c}$, and $\kk 1$ is the induced $\Vir_{\kk}^{\ge 0}$-module given by $L_n1 = 0$ for $n>0$, $L_01 = h1$, and $\mathbf{c}1 = c1$. There is a natural $\N$-grading on $V_{\Vir}(c,h)_{\kk}$, where the degree $n$ space admits a basis given by elements of the form
\begin{equation}\label{monomial}
    L_{-n_1}\cdots L_{-n_k}v_{c,h},
\end{equation}
where $v_{c,h} = 1\otimes 1$ and $n_1\ge \cdots \ge n_k\ge 1$ such that $\sum_in_i = n$. We refer to this as the \textit{standard basis of monomials}.
We call $V_{\Vir}(c,h)_{\kk}$ the \textit{Verma module}, and it is straightforward to see that $L_{m}$ acts on $V_{\Vir}(c,h)_{\kk}$ as a degree $-m$ operator.

It is known that
\[\overline{V}_{\Vir}(c,0)_{\kk} \eqdef V_{\Vir}(c,0)_{\kk}/U(\Vir_{\kk})L_{-1}v_{c,0}\]
is a VOA over $\kk$ with central charge $c$, vacuum element $\vac = v_{c,0}$, and conformal element $\omega = L_{-2}\vac$. 
As a VOA, it is generated by $\omega$ with
\[Y(\omega,z) = \sum_{k\in \Z}\omega_{(k)}z^{-1-k} = \sum_{m\in \Z}L_mz^{-2-m}.\]
That is, $\omega_{(k)} = L_{k-1}$.
We call this the \textit{universal Virasoro VOA} over $\kk$. Here, the overline on the $V$ will always denote quotienting by the submodule generated by $L_{-1}v_{c,0}$.
The degree $n$ part also admits a basis as a free $\kk$-module consisting of elements of the form \zcref[noname]{monomial} with $n_1\ge \cdots \ge n_k\ge 2$ such that $\sum_in_i = n$.

The $\kk$-module $V_{\Vir}(c,h)_{\kk}$ is an admissible module for this VOA for any $h\in \kk$. 
In fact, we have that $V_{\Vir}(c,h)_{\kk}$ is universal:
\[V_{\Vir}(c,h)_{\kk} = \PhiL(\kk v_{c,h}).\]
The quotient by $L_{-1}\vac$ guarantees that $\overline{V}(c,0)_\kk$ admits a VOA structure. Moreover, the element $L_{-1}\vac$ is a singular vector, which we will discuss later.

\begin{lemma}
    For all $n\in \Z$ and $m\ge 0$, we have
    \[(\ad_{L_1})^mL_{n} = (1-n)_mL_{n+m},\]
    where $(x)_m = x(x-1)\cdots (x-m+1)$ is the falling factorial.
\end{lemma}
For $m\ge 0$, we define the operator $L_1^{(m)}$ on a monomial of the form \zcref[noname]{monomial} as follows: We define
\[L_1^{(m)}\cdot L_n = \binom{1-n}{m}L_{n+m},\]
and we set $L_1^{(m)}$ to satisfy the generalized Leibniz rule on such a monomial:
\begin{equation}
    L_1^{(m)}(L_{-n_1}\cdots L_{-n_k}v_{c,h}) = \sum_{\substack{m_i\ge 0\\
    \sum_im_i = m}}(L_1^{(m_1)}\cdot L_{-n_1})\cdots (L_1^{(m_k)}\cdot L_{-n_k})v_{c,h}.
\end{equation}
This definition gives $\overline{V}(c,0)_\kk$ the structure of an $\mathcal{H}$-module vertex algebra such that $m!L_1^{(m)} = L_1^m$, so it is compatible with the VOA structure. Suppose there is an embedding of an intermediate subgroup $\Z\subseteq \Lambda\subseteq \Q$ into $\kk$. Suppose that $h\in \Lambda$ is a well-defined element in $\kk$. We may then give the Verma module $V_{\Vir}(c,h)_\kk$ the structure of a $\Lambda$-graded $(V,\mathcal{H})$-module. We will work with this structure more in \zcref{sec: rationality}.

Now we analyze the Zhu algebra. For various reasons, we will now restrict the above setting to the case where $\kk = \mathbb{D}$ is an integral domain.
\begin{notation}
    Let $\mathbb{D}$ be an integral domain containing $\Z$. Given a free $\mathbb{D}$-module $M_{\mathbb{D}}$, we denote by $M_{\Z}$ the free $\Z$-submodule of elements with integer coefficients with respect to the given basis.
\end{notation}
We may define $\overline{V}(c,0)_\Z$ to be the free $\Z$-module given by the standard basis of monomials of the form \zcref[noname]{monomial}. Note that $\overline{V}(c,0)_\Z$ is not naturally a vertex $\Z$-algebra since $c/2$ may not be in $\Z$. The base change $\overline{V}(c,0)_{\mathbb{D}} = \overline{V}(c,0)_{\Z}\otimes_\Z \mathbb{D}$ is a vertex $\mathbb{D}$-algebra by assumption, though.

Using the arguments in the appendix of \cite{Wang}, we obtain the following simplification of the Zhu algebra:
\begin{lemma}\label{lem: Zhu algebra arbitrary}
    Define $O'(\overline{V}(c,0)_{\mathbb{D}})$ to be the $\mathbb{D}$-submodule generated by elements of the form $\omega\circ^0_{0,n}v$ for all integers $n\ge 0$ and homogeneous $v\in \overline{V}(c,0)_{\mathbb{D}}$. Then $O'(\overline{V}(c,0)_{\mathbb{D}}) = O(\overline{V}(c,0)_{\mathbb{D}})$. Furthermore, we have isomorphisms of $\mathbb{D}$-modules
    \[\Aa(\overline{V}(c,0)_{\mathbb{D}}) \cong \mathbb{D}[x].\]
    Here $x$ is identified with the class $[\omega]\in A(\overline{V}(c,0)_{\mathbb{D}})$.
\end{lemma}
One may alternatively calculate the Zhu algebra by simply looking at the universal enveloping algebra definition.

Now, we relate $O(\overline{V}(c,0)_{\mathbb{D}})$ to $O(\overline{V}(c,0)_{\mathbb{F}})$, where $\mathbb{F} = \operatorname{Frac}(\mathbb{D})$ is the fraction field.
\begin{corollary}
    Let $\mathbb{D}$ be an integral domain, and let $\mathbb{F} = \operatorname{Frac}(\mathbb{D})$ be the fraction field. 
    Then
    \[O(\overline{V}(c,0)_{\mathbb{D}}) = O(\overline{V}(c,0)_{\mathbb{F}}) \cap \overline{V}(c,0)_{\mathbb{D}}.\]
    If $\mathbb{D}$ contains $\mathbb{Z}$, then the same is true if we replace $\mathbb{D}$ with $\mathbb{Z}$.
\end{corollary}
\begin{proof}
    Consider the map $\overline{V}(c,0)_{\mathbb{D}}\to \Aa(\overline{V}(c,0)_{\mathbb{D}})$. Both the domain and image are free $\mathbb{D}$-modules, so \zcref{torsion free lemma} gives us the first claim. The claim for when we replace $\mathbb{D}$ with $\mathbb{Z}$ follows from the same reasoning.
\end{proof}
It is important to note that the modules $\mathbb{D}v_{c,h}$ are self-dual over $\mathbb{D}$, so we may identify ${}^\theta(\mathbb{D}v_{c,h})^\lor \cong \mathbb{D}v_{c,h}$, where $\theta$ is the involution of $\overline{V}(c,0)_{\mathbb{D}}$ induced by the $\mathcal{H}$-module structure.
\begin{definition}
    We define $L_{\Vir}(c,h)_{\mathbb{D}}$ to be the quotient of $V_{\Vir}(c,h)_{\mathbb{D}}$ by the maximal graded $\overline{V}(c,0)_{\mathbb{D}}$-submodule $J\subset V_{\Vir}(c,h)_{\mathbb{D}}$ such that the degree zero part is equal to 0. That is,
    \[L_{\Vir}(c,h)_{\mathbb{D}} \eqdef {}_{\overline{V}_{\Vir}(c,0)_{\mathbb{D}}}\ThetaL({\mathbb{D}} v_{c,h}).\]
\end{definition}
The primary goal of this section is to understand these spaces when $\mathbb{D}=\mathbb{F}$ is a field.

\subsection{The discrete series}
Over $\C$, it is natural to consider the case of the discrete series central charges. We may study these over more general rings as well, and some results about the Zhu algebra carry over naturally.

Given coprime integers $r,s>1$, we define
\[c_{r,s} = 1-6\frac{(r-s)^2}{rs} \in \Q.\]
For $1\le m\le r-1$ and $1\le n\le s-1$, we define
\[h_{m,n} = \frac{(nr-ms)^2-(r-s)^2}{4rs} \in \Q.\]
We recall the following result on these elements, which is usually implicitly stated.
\begin{lemma}\label{lem:h set}
    Let $r,s>1$ be coprime integers. Define
    \[\widetilde h(r,s) = \{h_{m,n} \mid 1\le m\le r-1, 1\le n\le s-1\}.\]
    Denote by $h(r,s)$ the set of distinct elements in $\widetilde h(r,s)$. Then 
    \[|h(r,s)| = \frac{1}{2}|\widetilde h(r,s)| = \frac{1}{2}(r-1)(s-1).\]
\end{lemma}
Here are a number of interesting arithmetic facts about this setup.
\begin{lemma}
    Let $r,s>1$ be coprime. Then the denominator of $c_{r,s}$ after reducing fractions is given by
    \[\frac{rs}{\gcd(r,6)\gcd(s,6)}.\]
    Similarly, after reducing fractions, the denominator of $4h_{m,n}$ is given by
    \[\frac{rs}{\gcd(r,m^2-1)\gcd(s,n^2-1)}.\]
\end{lemma}
\begin{proof}
    After reducing fractions, the denominator will be $rs/\gcd(rs,6(r-s)^2)$. Since $r,s$ are coprime, we have
    \begin{align*}
        \gcd(rs,6(r-s)^2) &= \gcd(r,6(r-s)^2)\gcd(s,6(r-s)^2) \\
        &= \gcd(r,6s^2)\gcd(s,6r^2) \\
        &= \gcd(r,6)\gcd(s,6).
    \end{align*}
    This shows the first statement.
    Similarly, the denominator of $4h_{m,n}$ after reducing fractions is given by
    \[\frac{rs}{\gcd(rs,(nr-ms)^2-(r-s)^2)}.\]
    We have
    \begin{align*}
        &\gcd(rs,(nr-ms)^2-(r-s)^2) \\
        &= \gcd(r,(nr-ms)^2-(r-s)^2)\gcd(s,(nr-ms)^2-(r-s)^2)\\
        &= \gcd(r,(m^2-1)s^2)\gcd(s,(n^2-1)r^2)\\
        &= \gcd(r,m^2-1)\gcd(s,n^2-1).
    \end{align*}
    This shows the second statement.
\end{proof}
\begin{lemma}\label{lem:preimage c coprime}
    Let $p>3$ be prime. Then there are $(p+1)/2$ elements $c\in\mathbb{F}_p$ such that $c\equiv c_{r,s}$ for some coprime integers $r,s>1$.
\end{lemma}
\begin{proof}
    First, consider the map $f\colon \mathbb{F}_p^\times\to \mathbb{F}_p$ given by $x\mapsto x+x^{-1}$. Note that $f(x) = f(x^{-1})$. There are exactly two elements $x\in \mathbb{F}_p^\times$ with the property that $x = x^{-1}$, namely $1$ and $p-1$. For these, we have $f(1)=2$ and $f(p-1) = p-2$. Therefore, $|\im(f)| \le (p-3)/2 + 2 = (p+1)/2$. To see that this is an equality, suppose that $x+x^{-1} = y+y^{-1}$ for some $x,y\in \mathbb{F}_p^\times$. Rearranging, we have
    \[x^2y+y = y^2x+x \iff xy(x-y) - (x-y) = (xy-1)(x-y) = 0.\]
    Therefore, either $y = x$ or $x^{-1}$. This shows that $|\im(f)| = (p+1)/2$.

    Let $a \in \{0,\ldots,p-1\}$. Then $x\equiv (np+a)/(mp+1) \bmod p$, where $n,m\in \Z$. Given any $m>1$, there are infinitely many $n>1$ such that $\gcd(np+a,mp+1) = 1$ by the Chinese remainder theorem. Therefore, every element of $\mathbb{F}_p^\times$ is congruent to $r/s$ for coprime integers $r=np+a$ and $s = mp+1$. 
    Given $c = \mathbb{F}_p^\times$, let $a = (13-c)/6$. There are precisely $(p+1)/2$ many values of $a \in \mathbb{F}_p$ such that $a = x+x^{-1}$ for some $x\in \mathbb{F}_p^\times$. This shows the desired claim.
\end{proof}

We recall that a \textit{singular vector at level $N$} of $V_{\Vir}(c,h)_\kk$ is a homogeneous element $v_{c,h,N}$ of degree $N$ such that $L_iv_{c,h,N} = 0$ for all $i>0$. 
Over $\C$, it is known that the maximal graded submodule $J\subset V_{\Vir}(c_{r,s},h_{m,n})_\C$ is generated by two singular vectors at levels $mn$ and $(r-m)(s-n)$ respectively, which we denote as
\[v_{c_{r,s},h_{m,n},mn}\quad \text{and}\quad v_{c_{r,s},h_{m,n},(r-m)(s-n)}.\]
It is known that the coefficients of these are rational numbers with respect to the standard basis of monomials \cite{Watts24}, so we may regard the above singular vectors as elements of $V_{\Vir}(c_{r,s},h_{m,n})_\Q$.

In the case $h_{1,1} = 0$, we have that $v_{c_{r,s},0,1} = L_{-1}v_{c,0}$. Therefore, the quotient $\overline{V}(c_{r,s},0)_{\Q}$ is precisely the quotient of $V_{\Vir}(c,0)_{\Q}$ by one of its singular vectors. We denote by $v_{r,s}$ the image of the singular vector $v_{c_{r,s},0,(r-1)(s-1)}$ in $\overline{V}(c_{r,s},0)_{\Q}$.
Given $c\in \C$ such that $c\neq c_{r,s}$ for any $r,s>1$, the Verma module $V_{\Vir}(c,0)_\C$ only admits the singular vector $L_{-1}v_{c,0}$ \cite[Lemma 4.2]{Wang}. That is, the VOA $\overline{V}(c,0)_\C$ is simple provided that $c\neq c_{r,s}$ for any $r,s>1$ coprime. It is straightforward to see that this result descends to $\overline{V}(c,0)_\Q$ as well.

Such singular vectors are usually only determined up to a scalar coefficient. One standard normalization is to pick the coefficient of $L_{-1}^{N}v_{c_{r,s},h_{m,n}}$ in $v_{c_{r,s},h_{m,n},N}$ to be 1. Here is another way of picking the normalization, which is more convenient for us.
\begin{definition}
    Let $r,s>1$ be coprime. 
    Let $v_{c_{r,s},h_{m,n},N}$ be the singular vector at level $N$ in $V_{\Vir}(c_{r,s},h_{m,n})_\Q$, where $N = mn$ or $(r-m)(s-n)$. We denote by $\hat{v}_{c_{r,s},h_{m,n},N}$ the normalization of $v_{c_{r,s},h_{m,n},N}$ so that the coefficients of $\hat{v}_{c_{r,s},h_{m,n},N}$ with respect to the standard basis of monomials are integers with total gcd equal to 1. We define $\hat{v}_{r,s}$ in the same way for the singular vector $v_{r,s}$ in $\overline{V}(c_{r,s},0)_{\Q}$.
\end{definition}
In order for us to obtain the cleanest results possible, we work over the following rings:
\begin{equation}\label{eq:Rrs def}
    R_{r,s} = \Z[\tfrac{1}{2},c_{r,s},h(r,s)].
\end{equation}
That is, $R_{r,s}$ is the extension of $\Z$ by $\frac{1}{2}$, $c_{r,s}$, and all distinct possible values of $h_{m,n}$ for $1\le m\le r-1$ and $1\le n\le s-1$.
\begin{definition}
    We define $\widetilde{L}_{\Vir}(c_{r,s},h_{m,n})_{R_{r,s}}$ to be the quotient of the Verma module $V_{\Vir}(c_{r,s},h_{m,n})_{R_{r,s}}$ by the submodule generated by the two singular vectors at levels $mn$ and $(r-m)(s-n)$ defined above. For any ring morphism $R_{r,s}\to \kk$, we define $\widetilde{L}_{\Vir}(c_{r,s},h_{m,n})_{R}$ to be the quotient of $\overline{V}_{\Vir}(c_{r,s},h_{m,n})_{R}$ by the submodule generated by the singular vectors $\hat{v}_{c_{r,s},h_{m,n},N}$, where $N=mn$ and $N=(r-m)(s-n)$.
\end{definition}

Now, we analyze the Zhu algebra.
It is shown in \cite[Proposition 4.2]{Wang} that the image of $\hat{v}_{r,s}$ in the Zhu algebra $\Aa(\overline{V}(c_{r,s},0)_{\C})$ is a scalar multiple of the polynomial 
\[G_{r,s}(x) = \prod_{h\in h(r,s)}(x-h) \in \C[x],\] 
which is the unique square root of $\prod_{m=1}^{r-1}\prod_{n=1}^{s-1}(x-h_{m,n})$ up to a factor of $\pm 1$.
We define $\hat{G}_{r,s}$ to be the integer normalization of $G_{r,s}(x)$ so that the coefficients of the polynomial are integers with total gcd equal to 1.

The first and most important property is that the Zhu algebra is semisimple. First, we show this over $\Q$:
\begin{lemma}\label{singular vector rational}
    The image of the singular vector $\hat{v}_{r,s}\in \overline{V}(c_{r,s},0)_\Q$ in $\Aa(\overline{V}(c_{r,s},0)_\Q)$ is a nonzero rational multiple of the polynomial $G_{r,s}(x)$. Moreover,
    \[\Aa(\widetilde{L}_{\Vir}(c_{r,s},0)_\Q) \cong \Q[x]/(G_{r,s}(x)).\]
\end{lemma}
\begin{proof}
    We may identify $A(\overline{V}(c_{r,s},0)_\Q) = \Q[x]$, and the singular vector $v_{r,s}$ corresponds to a polynomial $H_{r,s}(x)\in \Q[x]$. In other words, the quotient $\widetilde{L}_{\Vir}(c_{r,s},0)_\Q$ has the Zhu algebra
    \[\Aa(\widetilde{L}_{\Vir}(c_{r,s},0)_\Q) = \Q[x]/(H_{r,s}(x)).\]
    By properties of base change, we have that
    \[\Aa(\widetilde{L}_{\Vir}(c_{r,s},0)_\Q)\otimes_\Q \C = \C[x]/(H_{r,s}(x)) \cong \Aa(\widetilde{L}_{\Vir}(c_{r,s},0)_\C) = \C[x]/(G_{r,s}(x)).\]
    As such, we conclude that $H_{r,s}(x)$ and $G_{r,s}(x)$ differ by a complex scalar in $\C[x]$. However, both of these polynomials have rational coefficients, so the complex scalar must be rational. This gives us the desired result.
\end{proof}
Now, we evaluate the rational scalar in \zcref{singular vector rational}.
\begin{lemma}\label{lem:Zhu alg integral}
    Modulo $O(\overline{V}(c_{r,s},0)_{\Z})\subset O(\overline{V}(c_{r,s},0)_{R_{r,s}})$, we have an identification
    \[\hat{v}_{r,s} \equiv \pm\hat{G}_{r,s}([\omega]).\]
    Moreover,
    \[\Aa(\widetilde{L}_{\Vir}(c_{r,s},0)_{R_{r,s}}) \cong R_{r,s}[x]/(\hat{G}_{r,s}(x)) = R_{r,s}[x]/(G_{r,s}(x)).\]
\end{lemma}
\begin{proof}    
    By \zcref{singular vector rational}, there exists a nonzero rational $\lambda = a/b$ such that 
    \[\hat{v}_{r,s} - \lambda \cdot \hat{G}_{r,s}([\omega]) \in O(\overline{V}(c_{r,s},0)_{\mathbb{Q}}).\]
    We may multiply by $b$ to obtain
    \[b\hat{v}_{r,s} - a \hat{G}_{r,s}([\omega]) \in O(\overline{V}(c_{r,s},0)_{\mathbb{Q}})\cap \overline{V}(c_{r,s},0)_{\mathbb{Z}} = O(\overline{V}(c_{r,s},0)_{\mathbb{Z}}).\]
    Equivalently, modulo $O(\overline{V}(c_{r,s},0)_{\mathbb{Z}})$ we have
    \[b\hat{v}_{r,s} \equiv a \hat{G}_{r,s}([\omega]) \mod O(\overline{V}(c_{r,s},0)_{\mathbb{Z}}).\]
    Without loss of generality, we may set $\gcd(a,b)=1$.
    Let $\ell$ be a prime dividing $a$. Then $\ell$ must divide the coefficients of $b\hat{v}_{r,s}$ in terms of the standard basis. We cannot have $\ell$ divide $b$ since $\gcd(a,b)=1$, so $\ell$ must divide the coefficients of $\hat{v}_{r,s}$. But this too is a contradiction since the gcd of the coefficients of $\hat{v}_{r,s}$ is 1. Therefore, no prime can divide $a$, so $a = \pm 1$. By the same argument, we conclude that $b = \pm 1$.
\end{proof}
\begin{remark}\label{rem: semisimple condition}
    Let $f\colon R_{r,s} \to \kk$ be any ring morphism, not just flat. By direct calculation, we can see that
    \[\Aa(\overline{V}_{\Vir}(c,0)_{\kk}) \cong \kk[x].\]
    By base change, one can see that the corresponding polynomial to the singular vector $\hat{v}_{r,s}$ is again $G_{r,s}(x)$. Therefore, for all such $\kk$ we have
    \[\Aa(\overline{V}_{\Vir}(c_{r,s},0)_{\kk}) \cong \kk[x]/(G_{r,s}(x)).\]
    This quotient is isomorphic to 
    \[\kk^{\frac{1}{2}(r-1)(s-1)} = \prod_{h\in h(r,s)}\kk\]
    in the case $\kk$ is a field and $G_{r,s}(x)$ is separable over $\kk$. Equivalently, $f(h_{m,n}) = f(h_{m',n'})$ holds in $\kk$ for some $m,n,m',n'$ if and only if $h_{m,n} = h_{m',n'}$ in $\Q$. However, this means the Zhu algebra is not semisimple unless $\kk$ is a field.
\end{remark}

Another interesting quantity to study is the algebra $R_V$.
First, we examine $R_{\overline{V}(c_{r,s},0)_{R_{r,s}}}$. By \zcref{lem:C2 generating set} and \zcref{lem: R_V surjection gr A(V)}, there are natural surjections
\[R_{r,s}[x]\twoheadrightarrow R_{\overline{V}(c_{r,s},0)_{R_{r,s}}} \twoheadrightarrow \gr\Aa(\overline{V}(c_{r,s},0)_{R_{r,s}}) \cong R_{r,s}[x],\]
where the last isomorphism is due to \zcref{lem: Zhu algebra arbitrary}. Letting $r$ be the generator of $R_{\overline{V}(c_{r,s},0)_{R_{r,s}}}$, the map $R_{\overline{V}(c_{r,s},0)_{R_{r,s}}} \twoheadrightarrow \gr\Aa(\overline{V}(c_{r,s},0)_{R_{r,s}})$ then corresponds to the map given by $r\mapsto x$. Therefore, 
\[R_{\overline{V}(c_{r,s},0)_{R_{r,s}}}\cong R_{r,s}[x]\]
because $R_{r,s}[x]\twoheadrightarrow R_{\overline{V}(c_{r,s},0)_{R_{r,s}}}$ admits a left inverse. 
\begin{lemma}\label{lem:R_V invertible coeff}
    The coefficient of $L_{-2}^{\frac{1}{2}(r-1)(s-1)}\vac$ in the singular vector $\hat{v}_{r,s} \in \overline{V}(c_{r,s},0)_{R_{r,s}}$ is invertible in $R_{r,s}$.
\end{lemma}
\begin{proof}
    The singular vector $\hat{v}_{r,s}$ in $\overline{V}(c_{r,s},0)_{R_{r,s}}$ can be written as a sum of monomials of the form $L_{-n_1}\cdots L_{-n_k}\vac$ with $n_1\ge \cdots \ge n_k\ge 2$. Any such monomial is in $C_2(\overline{V}(c_{r,s},0)_{R_{r,s}})$ if $n_1\ge 3$, so the image of $\hat{v}_{r,s}$ in $R_{\overline{V}(c_{r,s},0)_{R_{r,s}}}$ is precisely just the $L_{-2}^{\frac{1}{2}(r-1)(s-1)}\vac$ term in $\hat{v}_{r,s}$, which we denote by $a_{r,s}L_{-2}^{\frac{1}{2}(r-1)(s-1)}$. We have a surjective $R_{r,s}$-algebra morphism
    \[R_{r,s}[x]/(a_{r,s}x^{\frac{1}{2}(r-1)(s-1)}) \cong R_{\widetilde{L}_{\Vir}(c_{r,s},0)_{R_{r,s}}}\twoheadrightarrow \gr \Aa(\widetilde{L}_{\Vir}(c_{r,s},0)_{R_{r,s}}) \cong \mathbb{F}_p[x]/(x^{\frac{1}{2}(r-1)(s-1)}).\]
    From this, $a_{r,s}$ must be invertible in $R_{r,s}$. Indeed, the fact we have a surjective $R$-algebra map $R[x]/(ax^n) \to R[x]/(x^n)$ implies that $(ax^n)\supseteq (x^n)$. But $(x^n)\supseteq (ax^n)$ already, so $(x^n) = (ax^n)$. Therefore, there exists $b\in R$ such that $x^n=bax^n$. By cancellation, this implies $ba=1$, so $a\in R^\times$.
\end{proof}

\begin{proposition}\label{prop:C2 cofinite ring}
    Let $r,s>1$ be coprime, and let $R_{r,s}\to R$ be a flat ring morphism. Then $\widetilde{L}_{\Vir}(c_{r,s},0)_{R}$ is $C_2$-cofinite. More specifically, $R_{\widetilde{L}_{\Vir}(c_{r,s},0)_{R}}$ is free of rank $\frac{1}{2}(r-1)(s-1)$.
\end{proposition}
\begin{proof}
    Since $R_{r,s}\to R$ is flat, it suffices to show that $\widetilde{L}_{\Vir}(c_{r,s},0)_{R_{r,s}}$ is $C_2$-cofinite.

    From the above discussion and \zcref{lem:R_V invertible coeff}, we have that $R_{\widetilde{L}_{\Vir}(c_{r,s},0)_{R_{r,s}}}$ is isomorphic to the quotient of $R_{r,s}[x]$ by the polynomial $ax^{\frac{1}{2}(r-1)(s-1)}$, where $a$ is the integer coefficient of $L_{-2}^{\frac{1}{2}(r-1)(s-1)}\vac$ in the singular vector $\hat{v}_{r,s} \in \overline{V}(c_{r,s},0)_{R_{r,s}}$. However, this is invertible in $R_{r,s}$, so
    \[R_{\widetilde{L}_{\Vir}(c_{r,s},0)_{R_{r,s}}} \cong R_{r,s}[x]/(x^{\frac{1}{2}(r-1)(s-1)}) \cong (R_{r,s})^{\frac{1}{2}(r-1)(s-1)}.\]
\end{proof}
\begin{corollary}\label{Zhu bimods fin gen}
    Let $R_{r,s}\to \kk$ be an injective ring morphism with $\kk$ a field. Then $\widetilde{L}_{\Vir}(c_{r,s},0)_{\kk}$ is quasi-finite.
\end{corollary}
\begin{proof}
    Given an injective ring morphism $R_{r,s}\to \kk$ with $\kk$ a field, it must factor through the field of fractions $\operatorname{Frac}(R_{r,s}) = \Q$. Nonzero morphisms of fields are automatically flat and injective, and the morphism $R_{r,s}\to \operatorname{Frac}(R_{r,s})$ is flat because localization is flat. Therefore, $\widetilde{L}_{\Vir}(c_{r,s},0)_{\kk}$ is $C_2$-cofinite, and $R_{\widetilde{L}_{\Vir}(c_{r,s},0)_{\kk}}$ has dimension $\frac{1}{2}(r-1)(s-1)$ over $\kk$. By \zcref{lem:C2 cofinite implies quasi finite}, we are done.
\end{proof}
Note that the field $\kk$ in the above result must have characteristic 0 since it admits a morphism $\Q\to \kk$.

\subsection{Rationality in characteristic 0}\label{sec: rationality}
Now we consider the rationality of the simple quotient $L_{\Vir}(c,0)_{\mathbb{F}}$ over a field $\mathbb{F}$ of characteristic 0. For this, we will need the Kac determinant formula. Recall that the \textit{Shapovalov form} on $V_{\Vir}(c,h)_{\mathbb{F}}$ is the bilinear form $(\cdot,\cdot)$ uniquely determined by the following relations:
\begin{enumerate}
    \item $(v_{c,h},v_{c,h})=1$;
    \item $(L_{-n}v,w)=(v,L_{n}w)$ for all $v,w\in V_{\Vir}(c,h)_{\mathbb{F}}$;
    \item $(v,w)=0$ if $v$ and $w$ are homogeneous vectors of unequal degrees.
\end{enumerate}
The last property is technically redundant. By the tensor-hom adjunction, the Shapovalov form is equivalent to the morphism 
\[\psi^\lor_{\mathbb{F}v_{c,h}}\colon V_{\Vir}(c,h)_{\mathbb{F}} \to (V_{\Vir}(c,h)_{\mathbb{F}})',\] 
where the codomain is the contragredient dual as a module for $\overline{V}_{\Vir}(c,0)_{\mathbb{F}}$.

Now we recall the Kac determinant formula for the determinant of the Gram matrix of the Shapovalov form on the degree $N$ part of $V_{\Vir}(c,h)_{\mathbb{F}}$. We denote this by ${\det}_N(c,h)_{\mathbb{F}}$. This formula can be proven using a few different methods, but the one we will refer to involves algebraically approximating ${\det}_N(c,h)_{\mathbb{F}}$ as a polynomial in $h$. We will only give a sketch of the argument.

Given two polynomials $f(h),g(h) \in \kk[x]$, we write $f\sim_h g$ if we have $\deg_h(f-g) < \min\{\deg_hp,\deg_hq\}$. The polynomial ${\det}_N(c,h)_{\mathbb{F}}$ is by definition
\[{\det}_N(c,h)_{\mathbb{F}} = \det \left(\langle L_{-n_1}\cdots L_{-n_k}v_{c,h},L_{-m_1}\cdots L_{-m_\ell}v_{c,h}\rangle\right)_{\mathbf{n},\mathbf{m}}.\]
where $\mathbf{n} = (n_1,\ldots,n_k)$ and $\mathbf{m} = (m_1,\ldots,m_\ell)$ are tuples of integers such that 
\[n_1\ge \cdots\ge n_k\ge 1,\quad m_1\ge \cdots\ge m_\ell\ge 1,\quad \sum_in_i = \sum_jm_j = N.\]
Using the definition of the determinant, we have
\begin{align*}
    {\det}_N(c,h)_{\mathbb{F}} 
    &\sim_h \prod_{n_1\ge \cdots \ge n_k\ge 1}\langle L_{-n_1}\cdots L_{-n_k}v_{c,h},L_{-n_1}\cdots L_{-n_k}v_{c,h}\rangle \\
    &= \prod_{\substack{s_1 > \cdots > s_k\ge 1\\ r_1,\ldots,r_k\ge 1}}\langle L_{-s_1}^{r_1}\cdots L_{-n_k}^{r_k}v_{c,h},L_{-s_1}^{r_1}\cdots L_{-s_k}^{r_k}v_{c,h}\rangle.
\end{align*}
Using the bracket relations for the Virasoro, we obtain
\[{\det}_N(c,h)_{\mathbb{F}} \sim_h \prod_{\substack{r,s\ge 1\\ 1\le rs\le N}}\langle L_{-s}^rv_{c,h}, L_{-s}^rv_{c,h}\rangle^{m(N;r,s)},\]
where $m(N;r,s)$ is the number of partitions of $N$ such that $r$ appears exactly $s$ times. Explicitly,
\[m(N;r,s) = P(N-rs) - P(N-r(s+1)),\]
where $P(N)$ is the partition number.
Furthermore,
\[\langle L_{-s}^rv_{c,h}, L_{-s}^rv_{c,h}\rangle \sim_h (2rh)^ss!.\]
From this, we conclude that ${\det}_N(c,h)$ is a polynomial in $h$ of degree
\[\sum_{\substack{r,s\ge 1\\ 1\le rs\le N}}s\left(P(N-rs) - P(N-r(s+1))\right) = \sum_{\substack{r,s\ge 1\\ 1\le rs\le N}} P(N-rs),\]
and the highest degree term has coefficient
\begin{equation}\label{eq:A coeff}
    A_N = \prod_{\substack{r,s\in \N\\1\le rs\le N}}\left((2r)^ss!\right)^{P(N-rs)-P(N-r(s+1))}
\end{equation}
\begin{proposition}[Kac determinant formula]\label{prop:Kac determinant}
    Let $\mathbb{F}$ be a field of characteristic $0$. Then for all $N\ge 0$ we have
    \[{\det}_N(c,h)_{\mathbb{F}} = A_N\prod_{\substack{r,s\in \N\\
    1\le rs\le N}}(h-h_{r,s}(c))^{P(N-rs)},\]
    where $h_{r,s}(c)$ is given by
    \[h_{r,s}(c) = \frac{1}{48}\left((13-c)(r^2+s^2) + \sqrt{(c-1)(c-25)}(r^2-s^2) - 24rs - 2 + 2c \right).\]
\end{proposition}
\begin{proof}
    The proof is exactly the same as it is over $\C$. Since the determinant over $\C$ is a polynomial in $c$ and $h$ with coefficients in $\Q$, the formula is unchanged when working over any field $\mathbb{F}$ of characteristic 0. Due to symmetries in the product, the square root $\sqrt{(c-1)(c-25)}$ will always be paired with another factor of itself, so we will not actually be required to take square roots in $\mathbb{F}$.
\end{proof}
From this, we can show rationality.
\begin{theorem}\label{thm: rational over char 0}
Let $\mathbb{F}$ be a field of characteristic $0$, and let $c\in \mathbb{F}$. The simple quotient $L_{\Vir}(c,0)_{\mathbb{F}}$ is rational if and only if $c=c_{r,s}$ for some coprime integers $r,s>1$.
\end{theorem}
\begin{proof}
First, we consider the statement for $\mathbb{F}=\Q$.
By \zcref{singular vector rational}, we have
\[\Aa(\widetilde{L}_{\Vir}(c_{r,s},0)_\Q) = \Q[x]/(G_{r,s}(x)) \cong \Q^{\frac{1}{2}(r-1)(s-1)},\]
which is semisimple over $\Q$. If we have a simple $\Aa$-module $\mathsf{M}$, the induced module $\PhiL(\mathsf{M})$ is already quasi-rigid as it has finite-dimensional graded parts. It suffices to show that the maximal submodule $J_\Q\subseteq \PhiL(\mathsf{M})$ is 0. If we base change, then $\PhiL(\mathsf{M}\otimes_\Q\C)$ is a simple admissible module for $L_{\Vir}(c_{r,s},0)_\C$, so the maximal submodule $J_\C$ is 0. By maximality, we have
\[J_\Q\otimes_\Q\C \subseteq J_\C = 0,\]
so $J_\Q\otimes_\Q\C = 0$. The morphism $\Q\to \C$ is faithfully flat since it is an inclusion of fields, so this implies $J_\Q = 0$. Therefore, $\widetilde{L}_{\Vir}(c_{r,s},0)_{\Q} = L_{\Vir}(c_{r,s},0)_{\Q}$ is simple and rational by \zcref{DGK rational equiv}. 

By the second criterion for rationality from \zcref{DGK rational equiv}, the natural morphism $\psi_{\mathsf{M}}\colon \PhiL(\mathsf{M})\to \PhiL(\mathsf{M})'$ is an isomorphism for all simple $\Aa(\widetilde{L}_{\Vir}(c_{r,s},0)_\Q)$-modules (recall that all the simple modules of $\Aa(\widetilde{L}_{\Vir}(c_{r,s},0)_\Q)$ are self-dual). Using base change and \zcref{lem:base change dual}, the second criterion holds for the base change $\widetilde{L}_{\Vir}(c_{r,s},0)_{\mathbb{F}} = \widetilde{L}_{\Vir}(c_{r,s},0)_{\Q}\otimes_\Q \mathbb{F}$. The Zhu algebra of $\widetilde{L}_{\Vir}(c_{r,s},0)_{\mathbb{F}}$ is already semisimple by base change, so $\widetilde{L}(c_{r,s},0)_{\mathbb{F}} = L_{\Vir}(c_{r,s},0)_{\mathbb{F}}$ is simple and rational for all coprime integers $r,s>1$.

Now, we show that $L_{\Vir}(c,0)_{\mathbb{F}}$ is not rational if $c\neq c_{r,s}$ for all coprime integers $r,s>1$. By examining the Kac determinant formula for $V_{\Vir}(c,0)_{\mathbb{F}}$, the only singular vector of $V_{\Vir}(c,0)_{\mathbb{F}}$ is $L_{-1}v_{c,0}$. Therefore, the VOA $\overline{V}_{\Vir}(c,0)_{\mathbb{F}}$ is simple. However, it cannot be rational since the Zhu algebra is $\mathbb{F}[x]$ by \zcref{lem: Zhu algebra arbitrary}. Therefore, $L_{\Vir}(c,0)_{\mathbb{F}} = \overline{V}_{\Vir}(c,0)_{\mathbb{F}}$ is simple and not rational.
\end{proof}

\subsection{On rationality in positive characteristic}\label{sec: counterexample}
The question of rationality for Virasoro VOAs over a positive characteristic field is \textit{much} more nuanced. In \cite{Li19}, the authors analyzed the quotient of $\overline{V}(c,0)_{\mathbb{F}}$ by the vertex algebra ideal $I_0$ generated by the element $(L_{-n}^p-\delta_{p|n}L_{-np})\vac$ for all $n\ge 2$. We denote this quotient by $\overline{V}^0_{\Vir}(c,0)_{\mathbb{F}}$ (they denote it by $V^0_{\mathcal{V}\mathrm{ir}}(c,0)$), and we refer to it as the \textit{restricted Virasoro VOA}. As they show in \cite[Proposition 3.14]{Li19}, the ideal $I_0$ is generated by $L_{-2}^p\vac$. They showed that in $U(\Vir)$ the element $L_{n}^p-\delta_{p|n}L_{np}$ is central for all $n\in \Z$ \cite[Corollary 3.9]{Li19}.
They show in \cite[Theorem 4.4]{Li19} that the Zhu algebra is given by
\begin{equation}\label{eq:V^0 Zhu}
    \Aa(\overline{V}^0_{\Vir}(c,0)_{\mathbb{F}}) \cong \mathbb{F}[x]/(x^p-x).
\end{equation}
Note that $x^p-x$ is separable with roots $\ell\in \mathbb{F}_p$, where $\mathbb{F}_p$ is the finite field of order $p$. Therefore, $\mathbb{F}[x]/(x^p-x)$ is semisimple and finite-dimensional. They also show in \cite[Corollary 4.5]{Li19} that $\overline{V}^0_{\Vir}(c,0)_{\mathbb{F}}$ is $C_2$-cofinite.
\begin{notation}
    By the identity in \zcref[noname]{eq:V^0 Zhu}, the simple $\Aa(\overline{V}^0_{\Vir}(c,0)_{\mathbb{F}})$-modules are all one-dimensional, given by the action of $L_0$. Specifically, $L_0$ must act by one of the roots $\ell\in \mathbb{F}_p$. We denote the corresponding simple module as $\mathbb{F}v^0_{c,\ell}$, generated by the weight $\ell$ vector $v^0_{c,\ell}$.
\end{notation}
We will begin by studying $\overline{V}^0_{\Vir}(c,0)_{\mathbb{F}}$. The induced module functor $\PhiL$ is straightforward to characterize:
\begin{lemma}\label{lem: V^0 submodule}
    Let $\mathbb{F}$ be a field of characteristic $p>2$, and let $c,h\in \mathbb{F}$. Then the induced $\overline{V}^0_{\Vir}(c,0)_{\mathbb{F}}$-module $V^0_{\Vir}(c,h)_{\mathbb{F}} \eqdef \PhiL(\mathbb{F}v^0_{c,h})$ is isomorphic to the quotient of $V_{\Vir}(c,h)_{\mathbb{F}}$ by the submodule $W$ generated by the elements $(L_{-n}^p - \delta_{p|n}L_{-np})v_{c,h}$ for $n\ge 1$.
\end{lemma}
\begin{proof}
    By universal properties, the admissible $\overline{V}^0_{\Vir}(c,0)_{\mathbb{F}}$-module $\PhiL(\mathbb{F}v^0_{c,h})$ is isomorphic to the quotient of $V_{\Vir}(c,h)_{\mathbb{F}}$ by the submodule generated by the action of the vertex algebra ideal $I_0$. This is generated by elements of the form $a_{(n)}v^0_{c,h}$ for $a\in I_0$ and $n\in \Z$. 

    From \cite[Proposition 3.12]{Li19}, we have
    \begin{equation}
    \begin{aligned}
        Y((L_{-n}^p-\delta_{p|n}L_{-np})\vac,x) &= \sum_{j\ge 0}\binom{-n+1}{j}(-1)^j(L_{-n-j}^p - \delta_{p|(n+j)}L_{-(n+j)p})x^{jp}\\
        &\quad+\sum_{j\ge 0}\binom{-n+1}{j}(-1)^{-n-j}(L^p_{j-1}-\delta_{p|(j-1)}L_{(j-1)p})x^{(-n+1-j)p}.
    \end{aligned}
    \end{equation}
    Therefore, the modes $(L_{-2}^p\vac)_{(n)}$ are all central by \cite[Corollary 3.9]{Li19}.
    For all $a\in \overline{V}^0_{\Vir}(c,0)_{\mathbb{F}}$ and $m\in \Z$, we have by the associator formula
    \begin{align*}
        (a_{(m)}L_{-2}^p\vac)_{(n)} &= \sum_{i\ge 0}(-1)^i\binom{m}{i}\left(a_{(m-i)}(L_{-2}^p\vac)_{(n+i)} - (-1)^{m-1}(L_{-2}^p\vac)_{(m+n-i-1)}a_{(i)}\right)\\
        &= \sum_{i\ge 0}(-1)^i\binom{m}{i}\left(a_{(m-i)}(L_{-2}^p\vac)_{(n+i)} - (-1)^{m-1}a_{(i)}(L_{-2}^p\vac)_{(m+n-i-1)}\right).
    \end{align*}
    Here the second formula follows from the fact that, after applying the associator formula, $a_{(n)}$ can be written as an element of the completion of $U(\Vir_{\mathbb{F}})$, and $L_{-2}^p$ is central in $U(\Vir_{\mathbb{F}})$. The claim follows from this.
\end{proof}
One may ask if $\overline{V}^0_{\Vir}(c,0)_{\mathbb{F}}$ is rational. The answer, surprisingly, is no:
\begin{theorem}\label{thm:C2-cofinite semisimple Zhu not rational}
    Let $\mathbb{F}$ be a field of characteristic $p>2$, and let $c\in \mathbb{F}$. Then $\overline{V}^0_{\Vir}(c,0)_{\mathbb{F}}$ has a semisimple Zhu algebra and is $C_2$-cofinite, but it is not rational.
\end{theorem}
\begin{proof}
    We have by \cite[Theorem 4.4]{Li19} and \cite[Corollary 4.5]{Li19} that $\Aa(\overline{V}^0_{\Vir}(c,0)_{\mathbb{F}}) = \mathbb{F}[x]/(x^p-x)$ and $\overline{V}^0_{\Vir}(c,0)_{\mathbb{F}}/C_2(\overline{V}^0_{\Vir}(c,0)_{\mathbb{F}}) \cong  \mathbb{F}[x]/(x^p)$, so $\overline{V}^0_{\Vir}(c,0)_{\mathbb{F}}$ has semisimple Zhu algebra and is $C_2$-cofinite.
    By \zcref{rational iff An semisimple PhiL simple}, it suffices to show that one of the induced modules $\PhiL(\mathbb{F}v^0_{c,h})$ has a nontrivial singular vector of degree $\neq 0$. Once again, we denote by $W$ the kernel of the canonical quotient $V_{\Vir}(c,h)_{\mathbb{F}} \to \PhiL(\mathbb{F}v^0_{c,h})_{\mathbb{F}}$, which is the submodule generated by $(L_{-n}^p - \delta_{p|n}L_{-np})v_{c,h}$ for all $n\ge 1$ by \zcref{lem: V^0 submodule}.
    
    Since $L_{-n}^p - \delta_{p|n}L_{-np}$ is central, the element $(L_{-n}^p - \delta_{p|n}L_{-np})v_{c,h}$ is a singular vector of degree $np$ for all $n\ge 1$. Therefore, the dimension of the $j$-th graded component of $W$ is 0 for $0\le j< p$. This implies that $V_{\Vir}(c,h)_{\mathbb{F}}\to \PhiL(\mathbb{F}v^0_{c,h})$ is an isomorphism on the $j$-th graded component. In particular, the element $L_{-1}v^0_{c,h}$ is nonzero for all $h\in \mathbb{F}$. However, for $h=0$ we have that $L_{-1}v^0_{c,0}$ is a singular vector, so it generates a proper submodule of $\PhiL(\mathbb{F}v^0_{c,0})$. By \zcref{DGK rational equiv}, we conclude that $\overline{V}^0_{\Vir}(c,0)_{\mathbb{F}}$ is not rational since $\PhiL(\mathbb{F}v^0_{c,0})$ is not a simple admissible $\overline{V}^0_{\Vir}(c,0)_{\mathbb{F}}$-module despite $\mathbb{F}v^0_{c,0}$ being a simple $\Aa(\overline{V}^0_{\Vir}(c,0)_{\mathbb{F}})$-module.
\end{proof}
\begin{remark}
    As mentioned in the introduction, \cite[Main Theorem 3]{McRae_2026} states that over the complex numbers, an $\N$-graded, $C_2$-cofinite, simple, self-contragredient VOA over $\C$ with semisimple Zhu algebra is rational. It is natural to ask whether McRae's criterion extends to fields of positive characteristic. If $\overline{V}^0_{\Vir}(c,0)_{\mathbb{F}}$ were simple, then it would show McRae's rationality criterion does not extend. However, as we prove in \zcref{prop:not simple}, this is not the case. 
    
    One may instead consider the simple quotient $L_{\Vir}(c,0)_{\mathbb{F}}$, which is self-contragredient by \zcref{lem: simple module equiv} since all simple $\Aa(L_{\Vir}(c,0)_{\mathbb{F}})$-modules are one-dimensional and self-dual. Aside from rationality, this VOA is known to satisfy the other assumptions of McRae's assertion. While we do not pursue this here, it is natural to ask whether this VOA is rational, and moreover, whether McRae's criterion holds in positive characteristic.
\end{remark}

We showed above that $\overline{V}^0_{\Vir}(c,0)_{\mathbb{F}}$ is not rational by demonstrating that the associated functor $\PhiL$ does not send simple $\Aa$-modules to simple admissible modules. We will now refine this statement by calculating the levels of the nontrivial singular vectors of $\overline{V}^0_{\Vir}(c,0)_{\mathbb{F}}$ and the admissible modules $V^0_{\Vir}(c,h)_{\mathbb{F}}\eqdef \PhiL(\mathbb{F}v^0_{c,h})$.

Recall the Kac determinant formula for ${\det}_N(c,h)$ from \zcref{prop:Kac determinant}. Notice that the formula is well-defined and holds over any ring $\kk$ in which 2 and 3 are divisible (so that $1/48$ is well-defined). This is simply because the rational coefficients of the Kac determinant formula, after expanding in $c$ and $h$, are all elements of $\Z[\tfrac{1}{2},\tfrac{1}{3}]$. For what follows, set $\kk = \Z[\tfrac{1}{2},\tfrac{1}{3}]$. Pick a prime $p\ge 5$, and consider the quotient $V^0_{\Vir}(c,h)_{\kk}$ of $V_{\Vir}(c,h)_{\kk}$ by the submodule generated by $(L_{-n}^p - \delta_{p|n}L_{-np})v_{c,h}$ for all $n\ge 1$. The quotient is a free $\kk$-module whose basis is given by monomials of the form $L_{-s_1}^{r_1}\cdots L_{-s_k}^{r_k}v_{c,h}^0$, where $s_1>\cdots >s_k\ge 1$ and $1\le r_i<p$ for all $i$. There is a canonical mod $p$ map $\kk \to \mathbb{F}_p$, and the associated base change is 
\[V^0_{\Vir}(c,h)_{\kk}\otimes_\kk \mathbb{F}_p = V^0_{\Vir}(c,h)_{\mathbb{F}_p} = \PhiL(\mathbb{F}_pv^0_{c,h}).\]
This identity holds despite the fact that the map $\kk \to \mathbb{F}_p$ is not flat because $V^0_{\Vir}(c,h)_{\kk}$ is free over $\kk$.
Denote by ${\det}^0_{N}(c,h)$ the determinant of the Gram matrix for the \textit{induced Shapovalov form} on the degree $N$ component $V^0_{\Vir}(c,h)_{\kk,N}$. The determinant of the Gram matrix for the induced form on $V^0_{\Vir}(c,h)_{\mathbb{F}_p,N}$ is given by base change.

Using the above arguments for the original Kac determinant formula, we conclude that
\[{\det}_N^0(c,h)_{\kk} \sim_h \prod_{\substack{s\ge 1,\, 1\le r<p\\ 1\le rs\le N}}\langle L_{-s}^rv_{c,h}, L_{-s}^rv_{c,h}\rangle^{m(N;r,s)}.\]
Moreover, ${\det}_N^0(c,h)_{\kk}$ is a polynomial in $h$ of degree
\[\sum_{\substack{s\ge 1,\ 1\le r<p\\ 1\le rs\le N}}s\left(P(N-rs) - P(N-r(s+1))\right)\]
whose highest degree term has coefficient
\[A_N^0 \eqdef \prod_{\substack{s\ge 1,\ 1\le r<p\\ 1\le rs\le N}} ((2r)^ss!)^{P(N-rs) - P(N-r(s+1))}.\]
Consider its $p$-adic valuation $\nu_p(A_N^0)$, which is the exponent of the highest power of $p$ that divides $A_N^0$. We have $\nu_p(A_N^0) = 0$ for $N<p$, and $\nu_p(A^0_N) > 0$ otherwise. Therefore, the first generic singular vector is of degree $p$.

Consider the restricted VOA $\overline{V}^0_{\Vir}(c,0)_{\kk}$, which is the quotient of $V^0_{\Vir}(c,0)_{\kk}$ by the submodule generated by $L_{-1}v^0_{c,h}$. Applying the same reasoning as before, the determinant for the Gram matrix associated to the induced form on $\overline{V}^0_{\Vir}(c,0)_{\kk}$ is a polynomial in $c$ that is divisible by
\[\overline{A}_N^0 \eqdef \prod_{\substack{s\ge 2,\ 1\le r<p\\ 1\le rs\le N}} ((2r)^ss!)^{P(N-rs) - P(N-r(s+1))}.\]
The $p$-adic valuation is $\nu_p(\overline{A}_N^0) = 0$ for $N<2p$, and $\nu_p(\overline{A}_N^0) > 0$ for $N\ge 2p$. If we apply base change along a morphism $\kk\to \mathbb{F}$, where $\mathbb{F}$ is a field of characteristic $p$, then the resulting determinant will vanish mod $p$. We conclude with the following:
\begin{proposition}\label{prop:not simple}
    Let $\mathbb{F}$ be a field of characteristic $p\ge 3$. Then the VOA $\overline{V}^0_{\Vir}(c,0)_{\mathbb{F}}$ is not simple.
\end{proposition}
\begin{proof}
    By \zcref{lem: simple module equiv} and the fact that $\Aa(\overline{V}^0_{\Vir}(c,0)_{\mathbb{F}})$ is semisimple with one-dimensional simple modules, we have that $\overline{V}^0_{\Vir}(c,0)_{\mathbb{F}}$ is simple if and only if it is self-contragredient. The self-contragredient map
    \[\psi^\lor_{\mathbb{F}v_{c,0}^0}\colon \overline{V}^0_{\Vir}(c,0)_{\mathbb{F}} \to (\overline{V}^0_{\Vir}(c,0)_{\mathbb{F}})'\]
    corresponds to the induced Shapovalov form on $\overline{V}^0_{\Vir}(c,0)_{\mathbb{F}}$ by the tensor-hom adjunction.
    
    Applying base change to the above arguments, the determinant of the Gram matrix associated to the induced Shapovalov form on $\overline{V}^0_{\Vir}(c,0)_{\mathbb{F},N}$ is 0 for $N\ge 2p$. This implies that the induced Shapovalov form is degenerate, so $\overline{V}^0_{\Vir}(c,0)_{\mathbb{F}}$ is not self-contragredient, hence not simple.

    As for $p=3$, we used Mathematica to calculate the determinant of the Gram matrix. We found that the determinant at degree 18 vanishes mod 3, so for the same reasons $\overline{V}^0_{\Vir}(c,0)_{\mathbb{F}}$ is not simple.
\end{proof}
It is difficult to calculate the singular vectors mentioned in the above proof, even after applying base change.

\subsubsection{Computational work}\label{sec:computations}
Here, we provide some partial results on calculating the additional singular vectors that arise in positive characteristic. From this, we give some restrictions on the possible weights of the simple modules for the simple quotient $L_{\Vir}(c,0)_{\mathbb{F}}$.

\begin{remark}
    The discrete series rational numbers $c_{r,s}$ for coprime $r,s>1$ play an interesting role in the existence of singular vectors in positive characteristic.
    Suppose that $c \equiv c_{r,s}\bmod p$ for some coprime $r,s>1$. Then for all $h_{m,n}$ that are well-defined mod $p$, we can take the normalized singular vectors $\hat v_{c_{r,s},h_{m,n},mn}$ and $\hat v_{c_{r,s},h_{m,n},(r-m)(s-n)}$ to be singular vectors in $V(c_{r,s},h_{m,n})_{\mathbb{F}}$. In the case $m=n=1$, we have $h_{1,1}=0$. Then, $\overline{V}_{\Vir}(c_{r,s},0)_{\mathbb{F}}$ has a singular vector of degree $(r-1)(s-1)$. Furthermore, if $h_{m,n}=0$ mod $p$ for some $m,n$, then we will also have singular vectors of degrees $mn$ and $(r-m)(s-n)$ in $\overline{V}_{\Vir}(c_{r,s},0)_{\mathbb{F}}$.
\end{remark}

From \zcref{prop:not simple}, there are additional singular vectors in $\overline{V}^0_{\Vir}(c,0)_{\mathbb{F}}$ of degree $2p$ for $\Char\mathbb{F}=p\ge 5$. We will now provide the formulas for the singular vectors in characteristics 5 and 7 using Mathematica. 

We introduce the following notation. Given $n_1,\ldots,n_k\in \Z$, we define
\[L[n_1,\ldots,n_k] = L_{n_1}\cdots L_{n_k}\vac.\]
\begin{enumerate}
    \item In characteristic $p=5$, we have the degree 10 singular vector
    \begin{align*}
        v'_5 &= 4 (2 + c) L[-10] + L[-8, -2] + 4 L[-7, -3] + 4 c L[-6, -4] + (3 + c) L[-5, -5] \\
        &+ L[-6, -2, -2] + 2 L[-5, -3, -2] + L[-4, -4, -2] + L[-4, -3, -3].
    \end{align*}
    The image of this polynomial in the Zhu algebra $\Aa_0(\overline{V}^0_{\Vir}(c,0)_{\mathbb{F}})$ is
    \[f'_5(x) = 2 x (4 + x) (3 + 3 c + x).\]
    \item In characteristic $p=7$, we have the degree 14 singular vector
    \begin{align*}
        v'_7 
        &= 2 (2 + 5 c + c^2) L[-14] + c L[-12, -2] + 2 (3 + c) L[-11, -3] \\
        &+ 5 (4 + c) L[-10, -4] + 6 c L[-9, -5] + 2 (3 + c) (6 + c) L[-8, -6] \\
        &+ 5 (1 + c) (5 + c) L[-7, -7] + L[-10, -2, -2] + 5 L[-9, -3, -2] \\
        &+ 3 (6 + c) L[-8, -4, -2] + 6 (4 + c) L[-8, -3, -3] + 2 (2 + c) L[-7, -5, -2] \\
        &+ 6 (5 + c) L[-7, -4, -3] + 2 c L[-6, -6, -2] + 3 c L[-6, -5, -3] \\
        &+ 6 (2 + c) L[-6, -4, -4] + 6 L[-5, -5, -4] + 2 L[-8, -2, -2, -2] \\
        &+ 6 L[-7, -3, -2, -2] + 6 L[-6, -4, -2, -2] + 6 L[-6, -3, -3, -2] \\
        &+ 3 L[-5, -5, -2, -2] + 5 L[-5, -4, -3, -2] + 2 L[-5, -3, -3, -3] \\
        &+ 2 L[-4, -4, -4, -2] + 3 L[-4, -4, -3, -3].
    \end{align*}
    The image of this polynomial in the Zhu algebra $\Aa_0(\overline{V}^0_{\Vir}(c,0)_{\mathbb{F}})$ is
    \[f'_7(x) = 4 x (6 + x) (3 + 2 c + 3 c^2 + c x + x^2).\]
\end{enumerate}
We expect that there are additional independent singular vectors at every degree that is a multiple of $p$.

Consider the quotient of $\overline{V}^0_{\Vir}(c,0)_{\mathbb{F}}$ by the ideal $I$ generated by the above singular vectors for $\Char\mathbb{F}$. The resulting Zhu algebra is then isomorphic to the quotient of $\mathbb{F}[x]$ by the ideal generated by the gcd of $x^p-x$ together with $f'_p(x)$ in either of the above cases. We will restrict ourselves to $c\in \mathbb{F}_p$ in both cases.
\begin{enumerate}
    \item If $p=5$, then we consider
    \[g_5'(x) = {\gcd}_{\mathbb{F}_5[x]}\{x^5-x,f_5'(x)\} = {\gcd}_{\mathbb{F}_5[x]}\{x^p-x,2 x (4 + x) (3 + 3 c + x)\}.\]
    \begin{itemize}
        \item If $c=2 \in \mathbb{F}_5$, then $3+3c = 4$, so $g_5'(x) = x(4+x)$.
        \item If $c\in \mathbb{F}_5\setminus \{2\}$, then $4\neq 3+3c$, in which case $g_5'(x) = x(4+x)(3+3c+x)$.
    \end{itemize}
    \item If $p=7$, then we consider
    \[g_7'(x) = {\gcd}_{\mathbb{F}_7[x]}\{x^7-x,f_7'(x)\} = \gcd_{\mathbb{F}_7[x]}\{x^7-x,4 x (6 + x) (3 + 2 c + 3 c^2 + c x + x^2)\}.\]
    Let $h_c(x) = 3 + 2 c + 3 c^2 + c x + x^2$. Then we have the following cases:
    \begin{align*}
        h_0(x) &= (2+x)(4+x), &
        h_1(x) &= (3+x)(5+x),\\
        h_2(x) &= 5+3x+x^2, &
        h_3(x) &= 1+3x+x^2,\\
        h_4(x) &= (1+x)(3+x), &
        h_5(x) &= (1+x)(4+x),\\
        h_6(x) &= 4+6x+x^2.
    \end{align*}
    The polynomials $h_2(x)$, $h_6(x)$, and $h_3(x)$ do not have roots in $\mathbb{F}_7$, so in this case we have $g_7'(x) = x(6+x)$. Otherwise, we have $g_7'(x) = x(6+x)h_c(x)$ for $c=0,1,4,5$.
\end{enumerate}
We contextualize the above calculations as follows:
\begin{proposition}\label{prop:Virasoro calculations}
    Let $p=5,7$, and let $c\in \mathbb{F}_p$. let $\mathbb{F}$ be a field of characteristic $p$. Let $I$ be the vertex algebra ideal of $\overline{V}^0_{\Vir}(c,0)_{\mathbb{F}}$ generated by the singular vector of degree $2p$ described above. Then we have an isomorphism
    \[\Aa(\overline{V}^0_{\Vir}(c,0)_{\mathbb{F}}/I) \cong \mathbb{F}[x]/(g'_p(x)),\]
    where $g'_p(x)$ is the polynomial calculated above.
\end{proposition}
This allows us to further restrict the weights of the simple modules for the simple quotient $L_{\Vir}(c,0)_{\mathbb{F}}$:
\begin{corollary}
    Let $p=5,7$, and let $c\in \mathbb{F}_p$. Let $\mathbb{F}$ be a field of characteristic $p$. Let $f_c(x)$ be the unique polynomial up to rescaling by $\mathbb{F}^\times$ such that $\Aa(L_{\Vir}(c,0)_{\mathbb{F}}) \cong \mathbb{F}[x]/(f_c(x))$. Then $f_c(x)$ divides $g_p'(x)$. In particular, the simple admissible modules $L_{\Vir}(c,h)_{\mathbb{F}}$ for $L_{\Vir}(c,0)_{\mathbb{F}}$ must satisfy $g_p'(h)=0$.
\end{corollary}
\begin{question}
    Let $\mathbb{F}$ be a field of characteristic $p>2$, and let $c\in \mathbb{F}$. What is the kernel of the simple quotient map $V_{\Vir}(c,h)_{\mathbb{F}} \to L_{\Vir}(c,h)_{\mathbb{F}}$?
\end{question}
Note that the Zhu images of the two singular vectors calculated above are of the form $x(p-1+x)h_c(x)$, where $h_c(x)$ is a polynomial expression in $x$ and $c$ over $\mathbb{F}_p$. If $c$ is a transcendental element, then $h_c(x)$ has no roots in $\mathbb{F}_p$. As such, we conclude with the following conjecture:
\begin{conjecture}
    Let $\mathbb{F}$ be a field of characteristic $p\ge 5$ with a transcendental element $c\in \mathbb{F}$. Then $\Aa(L_{\Vir}(c,0)_{\mathbb{F}}) = \mathbb{F}[x]/(x(x-1))$.
\end{conjecture}

\bibliographystyle{amsalpha}
\bibliography{bibfile}

\end{document}